\documentclass[11pt]{article}
\usepackage[T1]{fontenc}
\usepackage{amsmath,amsthm,amsfonts,amssymb,graphicx,calc,microtype,thmtools,underscore,mathtools,anyfontsize,thm-restate,wrapfig}
\usepackage[shortlabels]{enumitem}
\setlist[itemize]{topsep=0ex,itemsep=0ex,parsep=0.4ex}
\setlist[enumerate]{topsep=0ex,itemsep=0ex,parsep=0.4ex}
\usepackage[usenames,dvipsnames,svgnames,table]{xcolor}
\usepackage[unicode=true]{hyperref}
\hypersetup{ 
colorlinks=true,
breaklinks=true,
linkcolor={black},
citecolor={black},
urlcolor={blue!60!black},
pdftitle={Quasi-tree-partitions of graphs with an excluded subgraph},
pdfauthor={Chun-Hung Liu, David~R.~Wood}} 
\usepackage[noabbrev,capitalise]{cleveref}
\crefname{lem}{Lemma}{Lemmas}
\crefname{thm}{Theorem}{Theorems}
\crefname{cor}{Corollary}{Corollaries}
\crefname{prop}{Proposition}{Propositions}
\crefname{conj}{Conjecture}{Conjectures}
\crefname{openproblem}{Open Problem}{Open Problems}
\crefname{claim}{Claim}{Claims}
\crefformat{equation}{(#2#1#3)}
\Crefformat{equation}{Equation #2(#1)#3}
\newcommand{\defn}[1]{\textcolor{Maroon}{\emph{#1}}}
\newcommand{\mathdefn}[1]{\textcolor{Maroon}{#1}}
\newcommand{\bigchi}{\raisebox{1.55pt}{\scalebox{1.2}{\ensuremath\chi}}}

\newcommand{\fchi}{\bigchi^f\hspace*{-0.2ex}}
\newcommand{\cchi}{\bigchi_{\star}\hspace*{-0.2ex}}
\newcommand{\lchi}{\bigchi^{\ell}}
\newcommand{\lcchi}{\bigchi^{\ell}_{\star}\hspace*{-0.2ex}}
\newcommand{\dchi}{\bigchi\hspace*{-0.1ex}_{\Delta}\hspace*{-0.2ex}}
\newcommand{\ldchi}{\bigchi\hspace*{-0.1ex}_{\Delta}^{\ell}\hspace*{-0.2ex}}

\newcommand{\cfchi}{\bigchi^f_{\star}\hspace*{-0.1ex}}
\newcommand{\dfchi}{\bigchi^f_{\Delta}\hspace*{-0.1ex}}

\newcommand{\TT}{\mathcal{T}}

\newcommand{\GG}{\mathcal{G}}

\newcommand{\CC}{\mathcal{C}}

\newcommand{\C}{\mathcal{C}}

\usepackage[longnamesfirst,numbers,sort&compress]{natbib}
\makeatletter
\def\NAT@spacechar{~}
\makeatother
\usepackage[margin=30mm]{geometry}
\allowdisplaybreaks

\DeclarePairedDelimiter{\ceil}{\lceil}{\rceil}

\renewcommand{\geq}{\geqslant}
\renewcommand{\leq}{\leqslant}
\DeclareMathOperator{\dist}{dist}
\DeclareMathOperator{\degen}{degen}

\DeclareMathOperator{\tw}{tw}

\DeclareMathOperator{\td}{td}
\DeclareMathOperator{\ltw}{ltw}
\DeclareMathOperator{\tpw}{tpw}
\renewcommand{\thefootnote}{\fnsymbol{footnote}}
\allowdisplaybreaks
\theoremstyle{plain}
\newtheorem{thm}{Theorem}
\newtheorem{lem}[thm]{Lemma}
\newtheorem{cor}[thm]{Corollary}
\newtheorem{prop}[thm]{Proposition}

\newtheorem{claim}{Claim}[thm]
\theoremstyle{definition}

\newcommand{\QQ}{\mathbb{Q}}
\newcommand{\NN}{\mathbb{N}}
\newcommand{\RR}{\mathbb{R}}
\begin{document}

\author{Chun-Hung Liu\,\footnotemark[2]
\qquad David~R.~Wood\,\footnotemark[3]}


\footnotetext[2]{Department of Mathematics, Texas A\&M University, USA (\texttt{chliu@tamu.edu}). Partially supported by NSF under CAREER award DMS-2144042.}

\footnotetext[3]{School of Mathematics, Monash   University, Melbourne, Australia  (\texttt{david.wood@monash.edu}). Research supported by the Australian Research Council and by NSERC.}

\sloppy

\title{\bf\boldmath Quasi-tree-partitions of graphs \\with an excluded subgraph}

\maketitle

\begin{abstract}
This paper studies the structure of graphs with given tree-width and excluding a fixed complete bipartite subgraph, which generalises the bounded degree setting. We give a new structural description of such graphs in terms of so-called quasi-tree-partitions. We demonstrate the utility of this result through applications to (fractional) clustered colouring. Further generalisations of these structural and colouring results are presented. 
\end{abstract}


\renewcommand{\thefootnote}
{\arabic{footnote}}

\section{Introduction}
\label{Intro}

Tree-width is a graph parameter that measures how similar a graph is to a tree. It is of fundamental importance in structural graph theory, especially Robertson and Seymour's Graph Minors project, and also in algorithmic graph theory, since many NP-complete problems are solvable in linear time on graphs with bounded tree-width. See \citep{HW17,Bodlaender98,Reed97} for surveys on tree-width. This paper studies the structure of graphs with given tree-width and excluding a fixed complete bipartite subgraph.

We now define tree-width\footnote{We consider simple finite undirected graphs $G$ with vertex-set $V(G)$ and edge-set $E(G)$. Let \defn{$\delta(G)$} and \defn{$\Delta(G)$} be the minimum degree and maximum degree of a graph $G$ respectively.}. For a tree $T$, a \defn{$T$-decomposition} of a graph $G$ is a collection $(B_x)_{x \in V(T)}$ such that:
\begin{itemize}
    \item $B_x\subseteq V(G)$ for each $x\in V(T)$, 
    \item for every edge ${vw \in E(G)}$, there exists a node ${x \in V(T)}$ with $\{v,w\} \subseteq B_x$, and 
    \item  for every vertex ${v \in V(G)}$, the set ${\{ x \in V(T) \colon v \in B_x \}}$ induces a non-empty (connected) subtree of $T$.
\end{itemize}
The \defn{width} of such a $T$-decomposition is ${\max\{ |B_x| \colon x \in V(T) \}-1}$. A \defn{tree-decomposition} is a $T$-decomposition for some tree $T$, denoted $(T,(B_x)_{x\in V(T)})$. The \defn{tree-width $\tw(G)$} of a graph $G$ is the minimum width of a tree-decomposition of $G$. Note that a connected graph has tree-width at most 1 if and only if it is a tree. Many graph classes have bounded tree-width, including outerplanar graphs, graphs of bounded circumference, etc. More generally, many graph classes $\GG$ have the property that $n$-vertex graphs in $\GG$ have treewidth $O(\sqrt{n})$, including planar graphs, graphs embeddable on a fixed surface, or graphs excluding a fixed minor. 

At the heart of this paper is the following question: What is the structure of graphs with tree-width $k$ (where we allow $k$ to depend on $|V(G)|$)? Without some additional assumption, not much more can be said beyond the definition. But with some additional assumption about an excluded subgraph, much more can be said about the structure of graphs with tree-width $k$.

We need the following definition. For graphs~$H$ and $G$, an \defn{$H$-partition} of~$G$ is a collection~$(V_x)_{x\in V(H)}$ such that:
\begin{itemize}
    \item $\bigcup_{x\in V(H)}V_x=V(G)$ and $V_x\cap V_y = \emptyset$ for all distinct $x,y\in V(H)$, and
    \item for each edge~${vw}$ of~$G$, if~${v \in V_x}$ and~${w \in V_y}$, then~${x = y}$ or~${xy \in E(H)}$.
\end{itemize}
The \defn{width} of such an $H$-partition is~${\max\{ |{V_x}| \colon x \in V(H)\}}$. 
Such partitions are related to graph product structure theory, since for any graph $H$, a graph $G$ has an $H$-partition of width at most $k$ if and only if $G$ is isomorphic to a subgraph of $H\boxtimes K_k$, where $\boxtimes$ denotes the strong product; see~\citep{UTW,DJMMW24,DEMWW22} for example.

Our starting point is \defn{tree-partitions} which are $T$-partitions where $T$ is a tree, as illustrated in \cref{TreePartition}. 
The \defn{tree-partition-width $\tpw(G)$} of a graph $G$ is the minimum width of a tree-partition of $G$. 
Tree-partitions were independently introduced by \citet{Seese85} and \citet{Halin91}, and have since been widely investigated \citep{Bodlaender-DMTCS99,BodEng-JAlg97, DO95,DO96,Edenbrandt86, Wood06,Wood09,BGJ22,DS20,Wood25}. Tree-partition-width has also been called \defn{strong tree-width} \citep{BodEng-JAlg97,Seese85}. 

\begin{figure}
    \centering
    \includegraphics{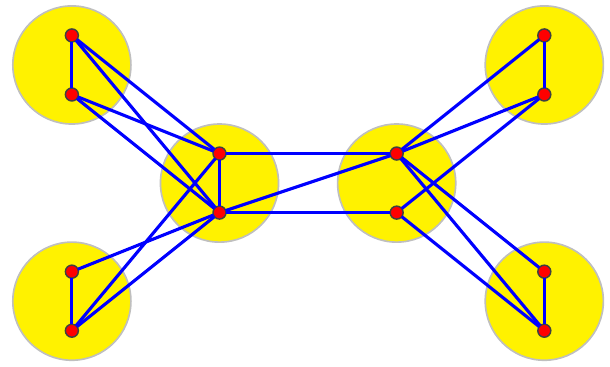}
    \caption{A tree-partition of width 2.}
    \label{TreePartition}
\end{figure}

It is easily seen that $\tw(G)\leq 2\tpw(G)-1$ for every graph $G$~\citep{Seese85}. 
But conversely, tree-partition-width cannot be upper bounded by a function of tree-width. 
For example, the $n$-vertex fan has tree-width 2 and tree-partition-width $\Theta(\sqrt{n})$, which was implicitly shown by \citet[(3.6)]{DO96}. On the other hand, tree-partition-width is upper bounded by a function of tree-width and maximum degree. In particular, a referee of a paper by \citet{DO95} showed the following:

\begin{thm}[\citep{DO95}] 
\label{TPW} 
For $k,\Delta\in\NN$, every graph with tree-width $k$ and maximum degree $\Delta$ has tree-partition-width $O(k\Delta)$.
\end{thm}

This result is incredibly useful, and has found applications in diverse areas including 
graph drawing~\citep{CDMW08,GLM05,DMW05,DSW07,WT07}, graphs of linear growth~\citep{CDGHHHMW23}, 
nonrepetitive graph colouring~\citep{BW08}, 
clustered graph colouring~\citep{ADOV03,LO18}, 
fractional fragility~\citep{DS20}, 
monadic second-order logic~\citep{KuskeLohrey05}, 
network emulations~\citep{Bodlaender-IPL88, Bodlaender-IC90, BvL-IC86, FF82}, 
statistical learning theory~\citep{ZA22}, 
size Ramsey numbers~\citep{DKCPS,KLWY21}, 
and the edge-{E}rd{\H{o}}s-{P}{\'o}sa property~\citep {RT17,GKRT16,CRST18}. 
The essential reason for the usefulness of \cref{TPW} is that in a tree-partition each vertex appears only once, unlike in a tree-decomposition. See \citep{UTW} for a generalisation of \cref{TPW} in terms of $H$-partitions. 

Note that the dependence on tree-width and maximum degree in \cref{TPW} is best possible up to a constant factor.  In particular,  \citet{Wood09} showed that for any $k\geq 3$ and sufficiently large $\Delta$ there is a graph with tree-width $k$, maximum degree $\Delta$, and tree-partition-width $\Omega(k\Delta)$.

\cref{TPW} has been extended in various ways. For example, \citet{DW24,DW22a} showed the same result where the underlying tree has maximum degree $O(\Delta(G))$ and at most $\max\{\frac{|V(G)|}{2k},1\}$ vertices. 

This paper explores structural descriptions of graphs with tree-width $k$ that satisfy some weaker assumption than bounded degree. A graph $G$ \defn{contains} a graph $H$ if some subgraph of $G$ is isomorphic to $H$. On the other hand, $G$ is \defn{$H$-subgraph-free} if no subgraph of $G$ is isomorphic to $H$. Our  focus is on $K_{s,t}$-subgraph-free graphs. The $s=1$ case corresponds to graphs with maximum degree less than $t$, and the $s=2$ case corresponds to graphs with codegree at most $t-1$.  \cref{TPW} shows that $K_{1,t}$-subgraph-free graphs with bounded treewidth have bounded tree-partition-width. However, this is false for $s=2$, since fan graphs are $K_{2,3}$-subgraph-free with treewidth 2 and unbounded tree-partition-width. 


The following more general lower bound, proved in \cref{LowerBound}, says that any result about $H$-partitions of $K_{2,t}$-subgraph-free graphs with bounded treewidth must allow for graphs $H$ that are much more general than trees. 

\begin{thm}
\label{ForceBadH}
For all $k,c,d\in\NN$ with $k\geq 2$ there is a $K_{2,k}$-subgraph-free graph $G$ with $\tw(G) \leq 2k-1$ such that for any graph $H$, if $G$ has an $H$-partition with width at most $c$, then $H$ contains $K_{\ceil{\sqrt{k}}}$ or $K_{\ceil{\sqrt{k}},d}$. 
\end{thm}


\cref{ForceBadH} says  there is no reasonably sparse graph $H$ so that $G$ has an $H$-partition with bounded width. In contrast, the main results of this paper show that we can remove a well-structured sparse subgraph of $G$ so that the remaining graph has a tree-partition.  To make this intuition precise, we need the following definitions, which are a key aspect of this paper. 

Consider a tree $T$ rooted at node $r\in V(T)$. 
For each node $x\in V(T)$, let 
$$\mathdefn{T\uparrow x} := \{y \in V(T) \colon \dist_T(r,y) < \dist_T(r,x)\},$$
where $\dist_T(r,y)$ is the \defn{distance in $T$} between $r$ and $y$, which is the minimum number of edges of a path in $T$ between $r$ and $y$.
As illustrated in \cref{QuasiTreePartition}, for $k\in\NN_0$ and a rooted tree $T$, a \defn{$k$-quasi-$T$-partition} of a graph $G$ is a pair
$\TT=((B_x)_{x\in V(T)},(E_v)_{v\in V(G)})$ where:
\begin{itemize}
\item $(B_x)_{x\in V(T)}$ is a $T$-partition 
of $G-\bigcup_{v\in V(G)}E_v$, and
\item for each $v\in V(G)$, $E_v$ is a set of at most $k$ edges in $G$ incident with $v$, and for each edge $vw\in E_v$, if $v\in B_x$ and $w\in B_y$ then $y\in T\uparrow x$. 
\end{itemize}
The \defn{width} of $\TT$ is $\max\{|B_x| \colon x\in V(T)\}$. 
The \defn{degree} of $\TT$ is $\Delta(T)$.
A \defn{$k$-quasi-tree-partition} is a $k$-quasi-$T$-partition for some rooted tree $T$. 
Note that $0$-quasi-tree-partitions are exactly tree-partitions.

\begin{figure}
    \centering
    \includegraphics{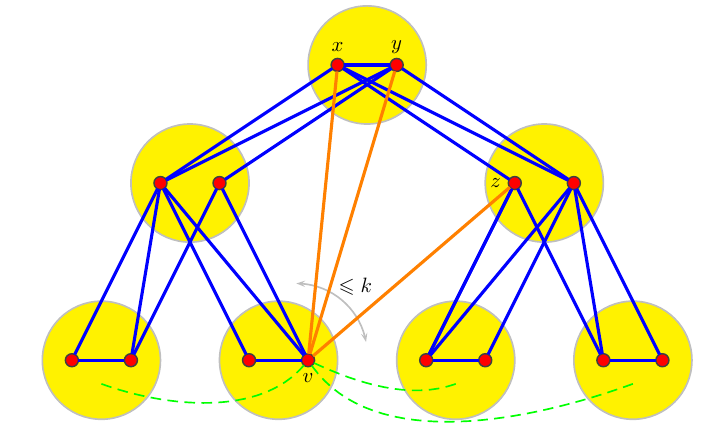}
    \caption{A $k$-quasi-tree-partition, where $E_v=\{vx,vy,vz\}$. Dashed edges are not allowed. Without edge $vz$, the quasi-tree-partition is clean.}
    \label{QuasiTreePartition}
\end{figure}

A $k$-quasi-$T$-partition $((B_x)_{x\in V(T)},(E_v)_{v\in V(G)})$ of a graph $G$ is \defn{clean} if for each edge $vw\in E_v$, if $v\in B_x$ and $w\in B_y$ then $y$ is a non-parent ancestor of $x$ in $T$. This strengthens the condition that  $y\in T\uparrow x$. 

Our main results apply to graphs satisfying a certain sparsity condition, which we now introduce. 
The \defn{1-subdivision} of a graph $H$ is the graph obtained from $H$ by replacing each edge $vw$ of $H$ by a path $vx_{vw}w$ internally disjoint from the rest of the graph. Let \defn{$\rho(G)$} be the maximum of $\delta(H)$ taken over all graphs $H$ such that the 1-subdivision of $H$ is a subgraph of $G$. 

We prove the following result for graphs of given tree-width that exclude a complete bipartite subgraph, where the `weight' term is explained in \cref{subsecWeight}. 

\begin{thm}
\label{Kst-rho}
For any $s,t,k,\rho\in\NN$, every $K_{s,t}$-subgraph-free graph $G$ with $\tw(G)\leq k$ and $\rho(G)\leq\rho$ has a clean $(s-1)$-quasi-tree-partition of width $O(t \rho^{s-1} k)$, degree $O(t \rho^{s-1})$, and weight $O(k)$.
\end{thm}

The main point of \cref{Kst-rho} is that the `quasiness' of the quasi-tree-partition depends only on $s$ (which we may assume is at most $t$ by symmetry, and is often much less than $k$). Moreover, we show that the ``$(s-1)$-quasi'' term in \cref{Kst-rho} is best possible (see \cref{ListColouring}).

In \cref{rho-tw} below, we show that $\rho(G)\leq\tw(G)$. Thus, 
\cref{Kst-rho} implies:

\begin{cor}
\label{KstBasic}
For any $s,t,k\in\NN$, every $K_{s,t}$-subgraph-free graph $G$ with tree-width at most $k$ has a clean $(s-1)$-quasi-tree-partition of width $O(tk^s)$, degree $O(tk^{s-1})$, and weight $O(k)$.
\end{cor}

A 0-quasi-tree-partition is precisely a tree-partition, and a graph is $K_{1,t}$-subgraph-free if and only if it has maximum degree less than $t$. So the $s=1$ case of \cref{KstBasic} implies \cref{TPW}, as well as matching the known degree bound from \citep{DW24,DW22a} (up to a constant factor). 


Since the dependence on $k$ in \cref{Kst-rho} is linear (for fixed $s,t,\rho$),  \cref{noKstStructure} gives interesting results even when tree-width is not bounded, such as for planar graphs and more generally for graphs $G$ of Euler genus $g$. It follows from Euler's formula that $G$ is $K_{3,2g+3}$-subgraph-free, and $\rho(G)\in O(\sqrt{g+1})$. Results about balanced separators by \citet{GHT-JAlg84} imply that $\tw(G)\in O(\sqrt{(g+1)|V(G)|})$. The next result thus follows directly from  \cref{Kst-rho}.

\begin{cor} 
\label{SurfaceCorllary}
For any $g\in\NN_0$ every graph $G$ with Euler genus $g$ has a clean $2$-quasi-tree-partition of width $O((g+1)^2 \tw(G)) = O((g+1)^{5/2}|V(G)|^{1/2})$, degree $O((g+1)^2)$, and weight $O(\tw(G)) = O((g+1)^{1/2}|V(G)|^{1/2})$. 
\end{cor}

We prove \cref{Kst-rho} in \cref{StructureI}, where we also give more motivating examples for minor-closed and non-minor-closed graph classes. 


\cref{KstBasic} has applications to clustered colouring and fractional clustered colouring, which we present in \cref{ColouringI}.

\cref{StructureII} presents a number of extensions and generalisations of \cref{Kst-rho,KstBasic} for (non-clean) quasi-tree-partitions, where we relax the $K_{s,t}$-subgraph-free assumption as follows. As illustrated in \cref{ExtensionSkewered}, for $s,t\in\NN$, a graph $G$ is a \defn{1-extension of $K_{s,t}$} if $G$ contains a connected subgraph $H$ such that contracting $H$ into a vertex creates a graph isomorphic to $K_{1,s,t}$. Note that $K_{s+1,t}$ is a minor of any 1-extension of $K_{s,t}$. A graph $G$ is a \defn{skewered $K_{s,t}$} if it can be obtained from $K_{s,t}$ by adding a path disjoint from the $s$-vertex side and passing through all vertices in the $t$-vertex side.  For example, the fan on $t+1$ vertices is a skewered $K_{1,t}$; it is a typical example of a $K_{2,t}$-subgraph-free graph whose tree-partition-width grows with $t$.


\begin{figure}[ht]
\centering
(a) \includegraphics{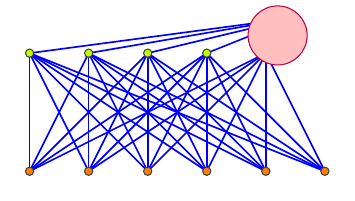}
\qquad
(b) \includegraphics{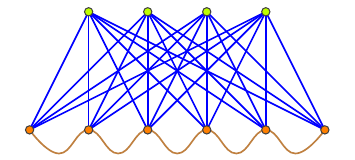}
\caption{(a) 1-extension of $K_{4,6}$, and (b) skewered $K_{4,6}$.}
\label{ExtensionSkewered}
\end{figure}

We prove the following extension of \cref{Kst-rho} (see \cref{NoExtensionSkewered-AlmostTreePartition} for a more precise statement).

\begin{thm}\label{NoExtensionSkeweredRhoCleanIntro}
For any $s,a,b,k,\rho \in\NN$ with $a,b\geq 2$, every graph $G$ with $\tw(G) \leq k$ and $\rho(G) \leq \rho$ that contains no 1-extension of $K_{s,a}$ and contains no skewered $K_{s,b}$ has a clean $(s-1)$-quasi-tree-partition of width $O(((s+ab)ab+k)\rho^{s-1}k)$. 
\end{thm}

The width in \cref{NoExtensionSkeweredRhoCleanIntro} can be improved by dropping the cleanness of the quasi-tree-partition.

\begin{thm}\label{NoExtensionSkeweredRhoIntro}
For any $s,a,b,k,\rho \in\NN$ with $a,b\geq 2$, every graph $G$ with $\tw(G) \leq k$ and $\rho(G) \leq \rho$ that contains no 1-extension of $K_{s,a}$ and contains no skewered $K_{s,b}$ has an $(s-1)$-quasi-tree-partition of width $O((s+ab)ab\rho^{s-1}k)$. 
\end{thm}

Recall that $\rho(G) \leq \tw(G)$ by \cref{rho-tw}. Thus \cref{NoExtensionSkeweredRhoIntro} implies the following extension of \cref{KstBasic} (see \cref{PseudoTW} for a more precise statement). 

\begin{cor} 
\label{NoExtensionSkewered}
For any  $s,a,b,k \in\NN$ with $a,b\geq 2$, every graph $G$ with $\tw(G)\leq k$ that contains no 1-extension of $K_{s,a}$ and contains no skewered $K_{s,b}$ has an $(s-1)$-quasi-tree-partition of width $O((s+ab)abk^{s})$. 
\end{cor}

This result with $s=1$ implies the following qualitative strengthening of the original result for tree-partitions (\cref{TPW}), since every graph with maximum degree $\Delta$ contains no 1-extension of $K_{1,\Delta+1}$ and contains no skewered $K_{1,\Delta+1}$

\begin{cor} 
\label{NoExtensionSkewered-TPW}
For any $a,b,k\in\NN$ every graph $G$ with $\tw(G)\leq k$ that contains no 1-extension of $K_{1,a}$ and no skewered $K_{1,b}$ has tree-partition-width $O(a^2b^2k)$.
\end{cor}




\subsection{Weight} \label{subsecWeight}

We now explain the `weight' term in the above results. Intuitively, the weight of a quasi-tree-partition is the `cost' of converting it into a tree-decomposition without extensively changing the ancestor-descendant relation of the bags. More formally, say $\TT=(T,(B_x)_{x\in V(T)},(E_v)_{v\in V(G)})$ is a clean quasi-tree-partition of a graph $G$. For each node $x\in V(T)$, define the \defn{load $C_x$} to be the set of vertices $w\in V(G)$, such that there exists an edge $vw\in E_v$ with $v\in B_y$ and $w\in B_a$, where $a$ is a non-parent ancestor of $x$ in $T$, and $y=x$ or $y$ is a descendant of $x$ in $T$. The \defn{weight} of $\TT$ is $\max\{|C_x|:x\in V(T)\}$. This property is interesting for the following reason. 
If $r$ is the root of $T$, then define $\widehat{B}_r:=B_r$. For each non-root node $x\in V(T)$ with parent $y\in V(T)$, define $\widehat{B}_x:=B_x\cup B_y\cup C_x$. Then $\widehat{T}:=(T,(\widehat{B}_x)_{x\in V(T)})$ is a tree-decomposition of $G$, since for each edge $vw\in E(G)$ with $v\in B_x$ and $w\in B_y$, where $y$ is an ancestor of $x$ or $y=x$, the vertex $w$ is in $\widehat{B}_q$ for each node $q$ on the $xy$-path in $T$. If $\TT$ has width $k$ and weight $w$, then $\widehat{T}$ has width at most $2k+w-1$, and $\tw(G)\leq 2k+w-1$. Thus, if $\TT$ comes from \cref{Kst-rho}, then $\widehat{T}$ has width $O(t \rho^{s-1} k)$. Without a bound on the weight of $\TT$, this construction of a tree-decomposition might not have bounded width. Indeed, as illustrated in \cref{Grid}, for all $n\in\NN$ there is a graph with tree-width $n$ that has a 1-quasi-tree-partition of width 1 and degree 2. 
A feature of \cref{Kst-rho} is that it produces a quasi-tree-partition that is not far from a tree-decomposition with small width.

\begin{figure}[!ht]
    \centering
    \includegraphics{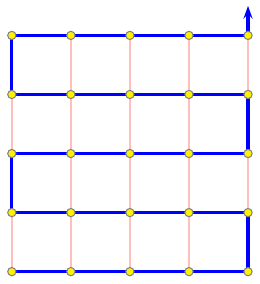}
    \caption{Order the vertices of the $n\times n$ grid (which has tree-width $n$) by following the blue path $P$. This defines a clean quasi-$P$-partition of width 1. Each vertex $v$ has at most one neighbour that is not a neighbour in $P$ and appears before $v$ in $P$. Thus this is a clean 1-quasi-$P$-partition.}
    \label{Grid}
\end{figure}

\section{Preliminaries}
\label{Prelims}

In this paper, $\NN:=\{1,2,\dots\}$ and $\NN_0:=\{0,1,2,\dots\}$.

We consider simple finite undirected graphs $G$ with vertex-set $V(G)$ and edge-set $E(G)$.

A \defn{rooted tree} is a tree $T$ with a nominated vertex called the \defn{root}. Consider a tree $T$ rooted at a vertex $r$. A path $P$ in $T$ is \defn{vertical} if the vertex in $P$ closest to $r$ is an endpoint of $P$. Consider a non-root vertex $x$ in $T$, and let $P$ be the $xr$-path in $T$. Every vertex $y$ in $P-x$ is an \defn{ancestor} of $x$, and $x$ is a \defn{descendent} of $y$. The neighbour $y$ of $x$ on $P$ is the \defn{parent} of $x$, and $x$ is a \defn{child} of $y$. The \defn{vertex-height} of a rooted tree $T$ is the maximum number of vertices on a root--leaf path in $T$. 

The \defn{closure} of a rooted tree $T$ is the graph with vertex-set $V(T)$, where two distinct vertices are adjacent if and only if one is an ancestors of the other in $T$. The \defn{tree-depth} of a graph $G$ is the minimum vertex-height of a tree $T$ such that $G$ is a subgraph of the closure of $T$. It is well-known and easily seen that $\tw(G)\leq\td(G)-1$.

A \defn{graph class} is a collection of graphs closed under isomorphism. A graph class $\GG$ is \defn{proper} if some graph is not in $\GG$. A graph class $\GG$ is \defn{hereditary} if for every graph $G$, every induced subgraph of $G$ is in $\GG$. A graph class $\GG$ is \defn{monotone} if for every graph $G\in\GG$, every subgraph of $G$ is in $\GG$.

A graph $H$ is a \defn{minor} of a graph $G$ if a graph isomorphic to $H$ can be obtained from $G$ by deleting vertices and edges, and contracting edges. A graph class $\GG$ is \defn{minor-closed} if for every graph $G\in \GG$, every minor of $G$ is in $\GG$. For example, for each $k\in\NN_0$, the class of graphs with tree-width at most $k$ is minor-closed.

A \defn{surface} is a compact 2-dimensional manifold. For any fixed surface $\Sigma$, the class of graphs embeddable on $\Sigma$ (without crossings) is minor-closed. A surface with $h$ handles and $c$ cross-caps has \defn{Euler genus} $2h+c$. The \defn{Euler genus} of a graph $G$ is the minimum Euler genus of a surface in which $G$ embeds.

Consider a graph $G$. For each vertex $v\in V(G)$, let \defn{$N_G(v)$} be the set of neighbours of $v$ in $G$. 
Now consider a set $X\subseteq V(G)$. 
For  $s\in\NN$, let \defn{$N^{\geq s}_G(X)$} be the set of vertices in $V(G)-X$ with at least $s$ neighbours in $X$. Let $\mathdefn{N_G(X)}:=N^{\geq 1}_G(X)$. 
A \defn{common neighbour} of $X$ is a vertex $v$ with $X\subseteq N_G(v)$; that is, $v\in N_G^{\geq |X|}(X)$.

Let $G$ be a graph. By considering a leaf bag, it is well-known and easily seen that, 
\begin{equation}
\label{delta-tw}
    \delta(G) \leq \tw(G).
\end{equation}
For $d\in\NN_0$, $G$ is \defn{$d$-degenerate} if $\delta(H)\leq d$ for every subgraph $H$ of $G$. The \defn{degeneracy $\degen(G)$} of $G$ is the maximum of $\delta(H)$, taken over all subgraphs $H$ of $G$. It is well-known and easily seen that 
\begin{equation}
\label{chi-degen-tw}
    \chi(G)-1 \leq \degen(G) \leq \tw(G).
\end{equation}
Note that degeneracy can be characterised via quasi-tree-partitions. 

\begin{prop}
\label{DegenQuasiTreePartition}
For $k\in\NN$, a graph $G$ is $k$-degenerate if and only if $G$ has a $(k-1)$-quasi-tree-partition of width 1. 
\end{prop}

\begin{proof}
$(\Longleftarrow)$ Assume $G$ has a $(k-1)$-quasi-tree-partition $(T,(B_x)_{x\in V(T)})$ of width 1. 
Given any subgraph $H$ of $G$, let $x$ be a deepest bag of $T$ such that $B_x\cap V(H)\neq\emptyset$, and let $v\in B_x\cap V(H)$. Let $y$ be the parent of $x$ in $T$. By the choice of $x$, $\deg_G(v)\leq |E_v|+|B_y|\leq k$. Hence $G$ is $k$-degenerate.

$(\Longrightarrow)$ We proceed by induction on $|V(G)|$. The base case is trivial. Let $G$ be a $k$-degenerate graph. So $G$ has a vertex $v$ of degree at most $k$. By induction, $G-v$ has a $(k-1)$-quasi-tree-partition of width 1. 
Let $y$ be a deepest node in $T$ such that $B_y\cap N_G(v)\neq\emptyset$; if $N_G(v)=\emptyset$. then let $y$ be the root of $T$. 
Add a child node $x$ of $y$ to $T$. Let $B_x:=\{v\}$. 
We obtain a $(k-1)$-quasi-tree-partition of $G$ of width 1, where $E_v:=\{vw\in E(G):w\in N_G(v)\setminus B_y\}$.
\end{proof}

Note that \cref{DegenQuasiTreePartition} generalises the construction of a quasi-tree-partition of the grid graph illustrated in \cref{Grid}. 

Note the following upper bound on $\rho(G)$. Say the 1-subdivision $H'$ of a graph $H$ is a subgraph of a graph $G$. Then $\delta(H)\leq\tw(H)$ by \cref{delta-tw}, and $\tw(H')\leq\tw(G)$ since tree-width is monotone under taking subgraphs. It is well-known and easily seen that $\tw(H)=\tw(H')$ (for any subdivision $H'$ of $H$). Combining these facts, 
$\delta(H)\leq\tw(H)=\tw(H')\leq\tw(G)$. Hence 
\begin{equation}
    \label{rho-tw}
    \rho(G)\leq\tw(G).
\end{equation}


We need the following variation on a result of \citet{OOW19}, where $K^*_{s,t}$ is the graph obtained from $K_{s,t}$ by adding $\binom{s}{2}$ new vertices, each adjacent to a distinct pair of vertices in the colour class of $s$ vertices in $K_{s,t}$, as illustrated in \cref{Kst}.

\begin{figure}[ht]
    \centering
    \includegraphics{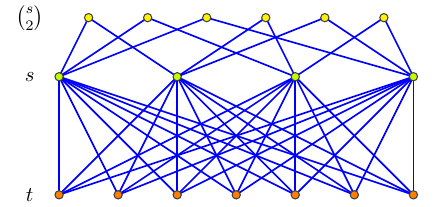}
    \caption{The graph $K^*_{4,7}$.}
    \label{Kst}
\end{figure}

\begin{lem}
\label{pi}
For any $s,t,\rho\in\mathbb{N}$, let
$$c:=c(s,t,\rho):=
    \begin{cases}
 t & \text{ when }s=1\\
  1+\rho + (t-1) \tbinom{\rho}{s-1} &  \text{ when }s\geq 2.
    \end{cases}$$
Then for every $K^*_{s,t}$-subgraph-free graph $G$ with $\rho(G)\leq\rho$, and for any set $X\subseteq V(G)$, 
$$| N^{\geq s}_G(X) |  \leq (c-1) \lvert X \rvert.$$
\end{lem}

\begin{proof}
If $s=1$ then every vertex in $X$ has degree at most $t-1$, implying $| N^{\geq s}_G(X) |  \leq (t-1) \lvert X \rvert = (c-1)\lvert X\rvert$. 
Now assume that $s\geq 2$. Let $H$ be the bipartite graph with bipartition $\{N^{\geq s}_G(X),\binom{X}{2}\}$, where $v\in N^{\geq s}_G(X)$ is adjacent in $H$ to $\{x,y\}\in\binom{X}{2}$ whenever $\{x,y\} \subseteq N_G(v)\cap X$. 
Let $M$ be a maximal matching in $H$. 
Let $Q$ be the graph with vertex-set $X$, where $xy\in E(Q)$ whenever $\{v,\{x,y\}\}\in M$ for some vertex $v\in N^{\geq s}_G(X)$. 
Thus, the 1-subdivision of every subgraph of $Q$ is a subgraph of $G$. Hence, $Q$ is $\rho$-degenerate, implying $|M| =  |E(Q)| \leq \rho |V(Q)| = \rho |X|$. Moreover,  $Q$ contains at most $\binom{\rho}{s-1}|X|$ cliques of size exactly $s$ \citep{NSTW06,Wood07}. Exactly $|M|$ vertices in $N^{\geq s}_G(X)$ are incident with an edge in $M$. 
For each vertex $v\in N^{\geq s}_G(X)$ not incident with an edge in $M$, by maximality, $N_G(v)\cap X$ is a clique in $Q$ of size at least $s$.  
Define a mapping from each vertex $v \in N^{\geq s}_G(X)$ to a clique of size exactly $s$ in $Q[N_G(v)\cap X]$.  At most $t-1$ vertices $v\in N^{\geq s}_G(X)$ are mapped to each fixed $s$-clique in $Q$, as otherwise $G$ contains $K^*_{s,t}$ (including the vertices matched to $\binom{X}{2}$). Hence $|N^{\geq s}_G(X)| \leq \rho|X| + (t-1) \binom{\rho}{s-1}|X| = (c-1)|X|$.
\end{proof}

Note that $c(s,t,\rho) \leq t\rho^{s-1}+1$.

Since $K^*_{s,t}$ contains $K_{s,t}$, any result for $K^*_{s,t}$-subgraph-free graphs is also applicable for $K_{s,t}$-subgraph-free graphs. On the other hand, there is only a small difference between $K^*_{s,t}$ and $K_{s,t}$ since $K_{s,t+\binom{s}{2}}$ contains $K^*_{s,t}$. The advantage in considering $K^*_{s,t}$ over $K_{s,t}$ is improved dependence on $t$. 

\section{Lower Bound}
\label{LowerBound}

The following result with $p=q=\ceil{\sqrt{k}}$ implies the lower bound, \cref{ForceBadH}, introduced in \cref{Intro}. The construction in fact has bounded tree-depth.

%
%
%
%
%

\begin{thm}
\label{ForceBadHpq}
For all $p,q,k,c,d\in\NN$ with $p,q\geq 2$ and $k \geq (p-1)(q-1)+1$, there is a $K_{2,k}$-subgraph-free graph $G$ with $\tw(G)+1 \leq \td(G) \leq 2k$ such that for every graph $H$, if $G$ has an $H$-partition with width at most $c$, then $H$ contains $K_p$ or $K_{q,d}$. 
\end{thm}

\begin{proof}
The following construction is illustrated in \cref{ConstructG}. 
Let $n\gg k,q,c,d$ as detailed below. 
Let $Q$ be the graph with vertex-set $[n]\times[k]$ where $(x,y)$ is adjacent to $(x',y')$ if and only if $x=x'$ and $y\neq y'$. Thus $Q$ consists of $n$ disjoint copies of $K_k$. For $x\in[n]$, the set $\{(x,y):y\in[k]\}$ is called a \defn{column} of $Q$. For $y\in[k]$, the set $\{(x,y):x\in[n]\}$ is called a \defn{row} of $Q$.
Let $T$ be the complete $n$-ary tree of vertex-height $k$. 
Initialise $G:=T$. 
For each root--leaf path $v_1,\dots,v_k$ in $T$,  
add to $G$ a copy of $Q$ disjoint from $G$, where $v_i$ is complete to the $i$-th row of $Q$, for each $i\in\{1,\dots,k\}$. 
Each vertex in $Q$ is called a \defn{descendent} of each of $v_1,\dots,v_k$. Delete $E(T)$ from $E(G)$. 

\begin{figure}[h]
    \centering
    \includegraphics{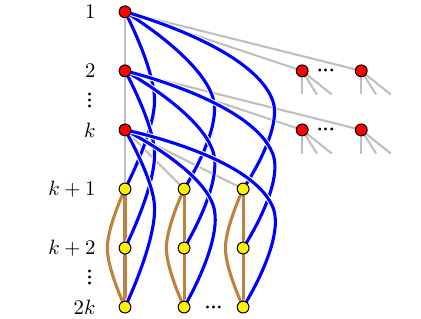}
    \caption{Construction of $G$. }
    \label{ConstructG}
\end{figure}

To see that $\td(G)\leq 2k$, let $T'$ be the tree obtained from the complete $n$-ary tree of vertex-height $k+1$ by subdividing each edge incident to a leaf $k-1$ times. So $T'$ has vertex-height $2k$, and $G$ is a subgraph of the closure of $T'$. Thus $\tw(G)+1\leq\td(G)\leq 2k$.

We now show that $G$ is $K_{2,k}$-subgraph-free. Label each vertex $v$ by $\dist_{T'}(r,v)+1$, where $r$ is the root of $T'$. Each label is in $[2k]$. 
There are two types of edges in $G$. 
A first type of edge joins vertices labelled $i$ and $k+i$ for some $i\in\{1,\dots,k\}$. 
A second type of edge joins vertices labelled $k+i$ and $k+j$ for some distinct $i,j\in\{1,\dots,k\}$. Suppose for the sake of contradiction that there are distinct vertices $v,w$ in $G$ with at least $k$ common neighbours. 
Say $x$ is a common neighbour of $v$ and $w$. 
First suppose that $v$ is labelled $i$ and $w$ is labelled $j$ for some $i,j\in\{1,\dots,k\}$. Thus, the edges $vx$ and $wx$ are of the first type, implying $x$ is labelled $k+i$ and $k+j$, implying $i=j$ and $v=w$ (since $v,w$ have a common neighbour $x$ not in $T$), which is a contradiction. 
Now suppose that $v$ is labelled $i$ and $w$ is labelled $k+j$ for some $i,j\in\{1,\dots,k\}$. Let $y$ be the unique ancestor of $w$ labelled $j$. Since $\deg_G(w)=k$, the $k$ common neighbours of $v$ and $w$ include $y$. Hence $vy$ is an edge joining vertices labelled $i$ and $j$, which is a contradiction. Finally, suppose that $v$ is labelled $k+i$ and $w$ is labelled $k+j$ for some $i,j\in\{1,\dots,k\}$. 
Let $y$ be the unique ancestor of $v$ labelled $i$.
Let $z$ be the unique ancestor of $w$ labelled $j$.
Since $\deg_G(v)=\deg_G(w)=k$, the $k$ common neighbours of $v$ and $w$ include both $y$ and $z$. Both $yv$ and $zv$ are edges of the first type, implying $i=j$ and $y=z$. The other neighbours of $v$ are precisely the vertices in the column containing $v$. Similarly, the other neighbours of $w$ are precisely the vertices in the column containing $w$. 
Since $i=j$ and $v\neq w$, these columns are distinct, implying that $v$ and $w$ have no common neighbours in addition to $y$. Since $k\geq 2$, this is a contradiction. Hence $G$ is $K_{2,k}$-subgraph-free. 


Let $H$ be any graph such that $G$ has an $H$-partition $(V_x)_{x\in V(H)}$ with width at most $c$. 
Let $v_1$ be the root of $T$. 
Say $v_1\in V_{x_1}$ where $x_1\in V(H)$. 
Since $|V_{x_1}|\leq c$ and $n\geq c$, there is a child $v_2$ of $v_1$ such that $v_2$ and every descendent of $v_2$ is not in $V_{x_1}$. 
Say $v_2\in V_{x_2}$. 
Since $|V_{x_2}|\leq c$ and $n\geq c$, there is a child $v_3$ of $v_2$ such that $v_3$ and every descendent of $v_3$ is not in $V_{x_2}$. 
Repeating this argument, we find a root--leaf path $v_1,\dots,v_k$ in $T$ such that $v_1,\dots,v_k$ are in distinct parts $V_{x_1},\dots,V_{x_k}$ and no descendent of $v_k$ is in $V_{x_1}\cup\dots\cup V_{x_k}$. 

Consider the copy of $Q$, where $v_i$ is complete to the $i$-th row of $Q$, for each $i\in[k]$. 
For each part $V_x$ (where $x\in V(H)$) that intersects $Q$, say the \defn{signature} of $V_x$ is the set of rows of $Q$ that $V_x$ intersects. Say $V_x$ is \defn{big} if the signature of $V_x$ has size at least $q$, otherwise $V_x$ is \defn{small}. Consider any set $S$ of $q$ rows in $Q$. If $d$ big parts have signature $S$, then $H$ contains $K_{q,d}$ (including the subset of $\{x_1,\dots,x_k\}$ that corresponds to $S$). Now assume that for each set $S$ of $q$ rows, there are less than $d$ big parts with signature $S$. Thus, the number of big parts that intersect $Q$ is less than $\binom{k}{q}d$. Each part intersects at most $c$ columns of $H$. Taking $n> \binom{k}{q}dc$, there is a column $X$ of $Q$ that intersects no big part. Hence, every part that intersects $X$ is small. Since the vertices in $X$ are in distinct rows, every small part intersects at most $q-1$ vertices in $X$. Thus, at least $\ceil{\frac{k}{q-1}}\geq p$ distinct parts intersect $X$, implying that $H$ contains $K_p$, as desired. 
\end{proof}

\section{Structural Results I}
\label{StructureI}

This section proves \cref{Kst-rho} from \cref{Intro}. \cref{Kst-rho} follows from the case $S=\emptyset$ in the following stronger result (\cref{noKstStructure}) which has some additional properties. In particular, Property (2) of \cref{noKstStructure} is used in \cref{StructureII}. The proof of \cref{noKstStructure} relies on a technical lemma (\cref{heart}), which is stated and proved below.

\begin{thm} 
\label{noKstStructure}
 Fix $s,t,\rho\in\mathbb{N}$ and define $c:=c(s,t,\rho)$ as in \cref{pi}. 
 Then for any $k\in\NN$ and any $K_{s,t}^*$-subgraph-free graph $G$ with $\tw(G)\leq k-1$ and $\rho(G)\leq\rho$, for any set $S\subseteq V(G)$ with $|S| \leq 12 ck$, there exists a clean $(s-1)$-quasi-tree-partition $(T,(B_x)_{x\in V(T)},(E_v)_{v\in V(G)})$ of $G$ of width at most $18ck$, degree at most $6c$, and weight at most $12k-1$, such that:
\begin{enumerate}[(1)]
	\item if $z$ is the root of $T$ then $S\subseteq B_z$, and
		\item for every $X \subseteq V(G)$, if $X$ has at least $k+1$ common neighbours in $G$, then there exists a vertical path in $T$ passing through each node $x \in V(T)$ with $X \cap B_x \neq \emptyset$.
\end{enumerate}
\end{thm}

\begin{proof}
First assume that $|V(G)|<4k$. Let $T$ be the tree with $V(T)=\{x\}$ and $E(T)=\emptyset$. Let $B_x:=V(G)$. Let $E_v:=\emptyset$ for each $v\in V(G)$. Then $((B_x)_{x\in V(T)},(E_v)_{v\in V(G)})$ is a clean $0$-quasi-$T$-partition. The width is $|B_x|=|V(G)|<4 k \leq 18ck$, the degree is $\Delta(T)=0\leq 6 c$, and the weight is 0. Property (2) is trivial. 

Now assume $|V(G)| \geq 4k$. If $|S|\geq 4k$ then we are done by \cref{heart} below. Otherwise, let $S' \subseteq V(G)$ with $S' \supseteq S$ and $|S'|=4k$. The result now follows from \cref{heart} replacing $S$ by $S'$.
\end{proof}

The next lemma is the heart of the paper. 

\begin{lem}
\label{heart}
Fix $s,t,\rho\in\mathbb{N}$ and define $c:=c(s,t,\rho)$ as in \cref{pi}. 
Then for any $k\in\NN$, for any $K^*_{s,t}$-subgraph-free graph $G$ with $\tw(G)\leq k-1$ and $\rho(G)\leq\rho$, for any set $S\subseteq V(G)$ with $4 k\leq|S| \leq 12 ck$, there exists a clean $(s-1)$-quasi-tree-partition $(T,(B_x)_{x\in V(T)},(E_v)_{v\in V(G)})$ of $G$ of width at most $18ck$, degree at most $6c$, and weight at most $12k-1$, 
such that: 
\begin{enumerate}[(1)]
\item if $z$ is the root of $T$, then $S\subseteq B_z$,  $|B_z|\leq \frac32|S|-2k$, and $\deg_T(z)\leq \frac{|S|}{2k} - 1$, and
\item for every $X \subseteq V(G)$, if $X$ has at least $k+1$ common neighbours in $G$, then there is a vertical path  in $T$ passing through each node $x \in V(T)$ with $X \cap B_x \neq \emptyset$,
\end{enumerate}
\end{lem}

\begin{proof}
We proceed by induction on the lexicographic order of $(|V(G)|,|V(G)\setminus S|)$.
	
\textbf{Case 1.} $|V(G) \setminus S|\leq 18ck$: Let $T$ be the 2-vertex tree with $V(T)=\{y,z\}$ and $E(T)=\{yz\}$. Consider $T$ to be rooted at $z$.  Note that $\Delta(T)=1\leq 6c$ and $\deg_T(z)=1\leq \frac{|S|}{2k}-1$. 
Let $B_z:=S$ and $B_y:=V(G) \setminus S$. 
Thus $|B_z|=|S|\leq \frac32 |S|-2k\leq 18ck$ and $|B_y|\leq |V(G) \setminus S|\leq 18ck$. 
Moreover, there exists a vertical path in $T$ passing through each node of $T$. Let $E_v:=\emptyset$ for each $v\in V(G)$. So the loads satisfy $C_z=C_y=\emptyset$, and the weight is 0.
Thus $((B_x)_{x\in V(T)},(E_v)_{v\in V(G)})$ is the desired clean ($s-1)$-quasi-$T$-partition of $G$.

Now assume that $|V(G) \setminus S|\geq 18ck$.
We first deal with the case when $|S|$ is small.

\textbf{Case 2.} $|V(G) \setminus S|\geq 18ck$ and $4k \leq |S|\leq 12 k-1$: Let $u$ be any vertex in $V(G) \setminus S$, and let $S':=S \cup\{u\} \cup N^{\geq s}_G(S \cup\{u\})$.	
By \cref{pi}, $|S'|\leq |S|+1+(c-1)( |S|+1)=c(|S|+1) \leq 12ck$. 
Since $|S'|>|S|\geq 4k$, by the induction hyposthesis, $G$ has a clean $(s-1)$-quasi-tree-partition $(T',(B'_x)_{x\in V(T')},(E'_v)_{v\in V(G)})$ of width at most $18ck$, degree at most $6c$, and weight at most $12k-1$ such that:
\begin{itemize}
\item  if $z'$ is the root of $T$, then $S'\subseteq B'_{z'}$, $|B'_{z'}|\leq \frac32|S'|-2k \leq 18ck$, $\deg_{T'}(z')\leq \frac{|S'|}{2k} - 1 \leq 6c-1$, and
\item for every $X \subseteq V(G)$, if $X$ has at least $k+1$ common neighbours in $G$, then there exists a vertical path in $T'$ passing through each node $x \in V(T')$ with $X \cap B'_x \neq \emptyset$.
\end{itemize}
Let $C'_x$ be the load of each $x\in V(T')$, so $|C'_x|\leq 12k-1$. 

Let $T$ be the tree obtained from $T'$ by adding one new node $z$ adjacent to $z'$. Consider $z$ to be the root of $T$. Let $B_z:=S$ and $B_{z'}:=B'_{z'}\setminus S$ and $B_x:=B'_x$ for each $x\in V(T')\setminus\{z'\}$. For $x\in V(T)$, we have $|B_x|\leq \max\{18ck,|S|\}\leq\max\{18ck,12k\}=18ck$. Hence the width bound is satisfied. Also, $S=B_z$ and $|B_z|=|S| \leq \frac32 |S|-2k$.

Let $E_v:=\emptyset$ for each $v\in B_z\cup B_{z'}$. 
For each child $x$ of $z'$ and for each $v\in B_x$, let $E_v:= \{vw\in E(G): w \in N_G(v) \cap B_z\}$, which has size at most $s-1$ (since $v\not\in B'_{z'} \supseteq S' \supseteq N^{\geq s}_G(S) = N^{\geq s}_G(B_z)$), and for each edge $vw\in E_v$, $w$ is in $B_y$ for some non-parent ancestor $y$ of $x$ in $T'$. 
For every other node $x$ of $T$ and for each $v\in B_x$, let $E_v:= E'_v$, which has size at most $s-1$, and for each edge $vw\in E_v$, $w$ is in $B_y$ for some non-parent ancestor $y$ of $x$ in $T'$.	
Since $B_z\subseteq B'_{z'}$ and each of $z$ and $z'$ has no non-parent ancestor, $\TT:=((B_x)_{x\in V(T)},(E_v)_{v\in V(G)})$ is a clean $(s-1)$-quasi-$T$-partition of $G$. 

Now consider the degree of $\TT$. By construction, $\deg_T(z)=1 \leq \frac{|S|}{2k} - 1$ and $\deg_{T}(z') = \deg_{T'}(z')+1\leq (6c-1) + 1 = 6c$. Every other vertex in $T$ has the same degree as in $T'$. Hence $\Delta(T)\leq 6c$, as desired. 

Now consider the weight of $\TT$. Observe that $C_z=C_{z'}=\emptyset$. For each child $x$ of $z'$ in $T$, we have $C_x\subseteq B_z=S$ and thus $|C_x|\leq|S|\leq 12k-1$. 
For every other node $x$ of $T$, we have $C_x=C'_x$ and thus $|C_x|=|C'_x|\leq 12k-1$. Hence $\TT$ has weight at most $12k-1$.


By construction, for every $X \subseteq V(G)$, if there exists a vertical path in $T'$ passing through each node $x \in V(T')$ with $X \cap B'_x \neq \emptyset$, then we can extend this path to be a vertical path in $T$ from $z$ such that it passes through all nodes $x \in V(T)$ with $X \cap B_x \neq \emptyset$. Hence (2) is satisfied. 

Now we deal with the last case in the proof.

\textbf{Case 3.} $|V(G) \setminus S|\geq 18ck$ and $12 k \leq |S|\leq 12ck$: By the separator lemma of \citet[(2.6)]{RS-II}, there are induced subgraphs $G_1$ and $G_2$ of $G$ with $G_1\cup G_2=G$ and $|V(G_1\cap G_2)|\leq k$, where $|S\cap V(G_i) \setminus V(G_{3-i})|\leq \frac23 |S|$ for each $i\in\{1,2\}$. Let $S_i := (S\cap V(G_i))\cup V(G_1\cap G_2)$ for each $i\in\{1,2\}$.
	
We now bound $|S_i|$. For a lower bound, since $|S\cap V(G_1) \setminus V(G_2)|\leq \frac23 |S|$, we have $|S_2|\geq |S \cap V(G_2)|\geq \frac13 |S| \geq 4k $. By symmetry, $|S_1|\geq  4k $. For an upper bound, $|S_i|\leq\frac23 |S| + k \leq 8ck + k \leq 12ck$. Also note that $|S_1|+|S_2|\leq |S|+2k$. 
 
We have shown that $4k \leq |S_i|\leq 12ck$ for each $i\in\{1,2\}$. Thus, we may apply induction to $G_i$ with $S_i$ the specified set. 
Hence $G_i$ has a clean $(s-1)$-quasi-tree-partition $\TT^i:=(T_i,(B^i_x)_{x\in V(T_i)},(E^i_v)_{v\in V(G_i)})$ of width at most $18ck$, degree at most $6c$, and weight at most $12k-1$ such that:
 \begin{itemize}
		\item  if $z_i$ is the root of $T_i$ then $S_i\subseteq B^i_{z_i}$,  $|B^i_{z_i}|\leq \frac32|S_i|-2k$, and $\deg_{T_i}(z_i)\leq \frac{|S_i|}{2k} - 1$, and
		\item for every $X \subseteq V(G_i)$, if $X$ has at least $k+1$ common neighbours in $G_i$, then there is a vertical path in $T_i$ passing through each node $x \in V(T_i)$ with $X \cap B^i_x \neq \emptyset$.
	\end{itemize}
 Let $C^i_x$ be the load in $\TT^i$ of each node $x\in V(T_i)$, so $|C^i_x|\leq 12k-1$. 

Let $T$ be the tree obtained from the disjoint union of $T_1$ and $T_2$ by merging $z_1$ and $z_2$ into a node $z$. Consider $T$ to be rooted at $z$. Let $B_z:= B^1_{z_1}\cup B^2_{z_2}$. Let $B_x:= B^i_x$ for each $x\in V(T_i)\setminus\{z_i\}$. 

 Let $E_v:=E^i_v$ for each $v\in V(G_i)$. This is well-defined since $V(G_1 \cap G_2) \subseteq B^1_{z_1} \cap B^2_{z_2}$, implying $E^1_v=E^2_v=\emptyset$ for each $v\in V(G_1\cap G_2)$. 
 By construction, for each $v\in V(G)$, if $v\in B_x$ for some $x \in V(T)$, then $E_v$ is a set of at most $s-1$ edges of $G$ incident with $v$, and for each edge $vw\in E_v$, $w \in B_y$ for some non-parent ancestor $y$ of $x$ in $T$.
 Since $G=G_1\cup G_2$ and $V(G_1\cap G_2)\subseteq B^1_{z_1}\cap B^2_{z_2} \subseteq B_z$, we have that $\TT:=((B_x)_{x\in V(T)},(E_v)_{v\in V(G)})$ is a clean $(s-1)$-quasi-$T$-partition of $G$.  
 
 Consider the width of $\TT$. By construction, $S\subseteq B_z$ and since $V(G_1\cap G_2)\subseteq B^i_{z_i}$ for each $i \in [2]$, 
	\begin{align*}
		|B_z| 
		& \leq |B^1_{z_1}|+|B^2_{z_2}| - |V(G_1\cap G_2)|\\
		& \leq (\tfrac32|S_1|-2k) +  (\tfrac32|S_2|-2k) - |V(G_1\cap G_2)|\\
		& = \tfrac32( |S_1|+ |S_2|) -4k - |V(G_1\cap G_2)|\\
		& \leq \tfrac32( |S| + 2|V(G_1\cap G_2)| ) -4k - |V(G_1\cap G_2)|\\
		& \leq \tfrac32 |S| + 2|V(G_1\cap G_2)| -4k\\
		& \leq \tfrac32 |S| - 2 k\\
		& < 18 ck.
	\end{align*}
Every other part has the same size as in $\TT^1$ or $\TT^2$. 
So $|B_x|\leq 18ck$ for each $x\in V(T)$. 

 Now consider the degree of $\TT$.  Note that
	\begin{align*}
		\deg_T(z)   = \deg_{T_1}(z_1) + \deg_{T_2}(z_2)
		& \leq  (\tfrac{|S_1|}{2k}-1) + (\tfrac{|S_2|}{2k}-1)\\
		& =  \tfrac{|S_1|+|S_2|}{2k} -2\\
		& \leq  \tfrac{|S|+2k}{2k} -2\\
		& =  \tfrac{|S|}{2k} -1\\
		& < 6 c. 
	\end{align*}
	Every other node of $T$ has the same degree as in $T_1$ or $T_2$. 
	Thus $\Delta(T) \leq 6c$. 

Finally, consider the weight of $\TT$. Let $C_x$ be the load of each $x\in V(T)$. Since $z$ is the root, $C_z=\emptyset$. 
For each node $x\in V(T_i)\setminus\{z_i\}$, we have $C_x=C^i_x$, so $|C_x|\leq 12k-1$. Hence $\TT$ has weight at most $12k-1$.

Let $X$ be an arbitrary subset of $V(G)$ such that nodes in $X$ have at least $k+1$ common neighbours in $G$.
If $X \setminus V(G_1) \neq \emptyset \neq X \setminus V(G_2)$, then all of the at least $k+1$ common neighbours of $X$ are contained in $V(G_1 \cap G_2)$, which is a set with size at most $k$, a contradiction.
Hence there exists $i \in [2]$ such that $X \subseteq V(G_{i})$.
So every node $x \in V(T)$ with $X \cap B_x \neq \emptyset$ is a node of $T_{i}$ with $X \cap B^i_x \neq \emptyset$.
If $X$ has at least $k+1$ common neighbours in $G_{i}$, then there exists a vertical path in $T_{i}$ (and hence in $T$) passing through each node $x \in V(T)$ with $X \cap B_x \neq \emptyset$.
If $X$ does not have at least $k+1$ common neighbours in $G_{i}$, then at least one common neighbour of $X$ is contained in $V(G_{3-i}) \setminus V(G_{i})$, so $X \subseteq V(G_1 \cap G_2) \subseteq B_z$, so the path consisting of $z$ is a vertical path in $T$ passing through each node $x \in V(T)$ with $X \cap B_x \neq \emptyset$.
This completes the proof.
\end{proof}

The next result is a more precise version of  \cref{KstBasic}. It follows from \cref{noKstStructure} and \cref{rho-tw} since
 $\rho(G)\leq \tw(G)\leq k-1$, implying $c \leq tk^{s-1}$.

\begin{cor}
\label{noKstStructureTreewidthk}
For any $k,s,t\in\NN$, any $K^*_{s,t}$-subgraph-free graph $G$ with $\tw(G)\leq k-1$ has a clean $(s-1)$-quasi-tree-partition of width at most $18 t k^s$, degree at most $6 tk^{s-1}$, and weight at most $12k-1$. 
\end{cor}

\cref{noKstStructureTreewidthk} with $s=1$  recovers \cref{TPW} for tree-partitions of graphs with given tree-width and maximum degree, and also recovers the degree bound in \citep{DW24,DW22a}. 

We now give more examples of \cref{Kst-rho}. First consider $K_t$-minor-free graphs $G$. Here $G$ is $K_{t-1,t-1}$-subgraph-free, since contracting a matching in $K_{t-1,t-1}$ with size $t-2$ gives $K_t$. 
It follows from a result of \citet{Kostochka82,Kostochka84} and \citet{Thomason84,Thomason00} that $\rho(G)\in O(t\sqrt{\log t})$. And  \citet{AST-SJDM94} showed that $\tw(G)\leq t^{3/2} |V(G)|^{1/2}$. 
The next result thus follows directly from the $S=\emptyset$ case of \cref{noKstStructure}.

\begin{cor}
For any $t\in\NN$ there exists $c_1,\dots,c_5\in\NN$ such that every $K_t$-minor-free graph $G$ has a clean $(t-2)$-quasi-tree-partition of width at most $c_1\tw(G) \leq c_2 |V(G)|^{1/2}$, degree at most $c_3$, and weight at most $c_4 \tw(G) \leq c_5 |V(G)|^{1/2} $.
\end{cor}

\cref{Kst-rho} is also applicable and interesting for non-minor-closed classes. 
The following definitions by \citet{DMW17} are useful for this purpose. 
A \defn{layering} of a graph $G$ is an ordered partition $(V_1,V_2,\dots)$ of $V(G)$ such that for each edge $vw\in E(G)$, if $v\in V_i$ and $w\in V_j$ then $|i-j|\leq 1$. The \defn{layered tree-width $\ltw(G)$} of a graph $G$ is the minimum $k\in\NN$ such that $G$ has a layering $(V_1,V_2,\dots)$ and a tree-decomposition $(T,(B_x)_{x\in V(T)})$ such that $|V_i\cap B_x|\leq k$ for each $i\in\NN$ and $x\in V(T)$. For example, \citet{DMW17} proved that every planar graph has layered tree-width at most 3; more generally, every graph with Euler genus $g$ has layered tree-width at most $2g+3$; and most generally, a minor-closed class $\GG$ has bounded layered tree-width if and only if some apex graph is not in $\GG$. \citet[Lemma~8]{DMW17} noted in their proof that for every graph $G$, 
\begin{align}
    \label{delta-ltw}
    \delta(G)\leq 3\ltw(G)-1.
\end{align}
The next lemma is proved using an idea from \citep[Lemma~9]{DMW17}.

\begin{lem}
\label{ltw-rho}
For every graph $G$, $$\rho(G)\leq 6\ltw(G)-1.$$
\end{lem}

\begin{proof}
Let $k:=\ltw(G)$. Say $H$ is a graph such that the 1-subdivision $H'$ of $H$ is a subgraph of $G$. So $\ltw(H')\leq k$. Consider a layering and tree-decomposition of $H'$ such that each layer has at most $k$ vertices in each bag. 
For each vertex $x$ of $H'$ obtained by subdividing an edge $vw$ of $H$, replace each instance of $x$ in the tree-decomposition of $H'$ by $v$. We obtain a tree-decomposition of $H$ without increasing the bag size. In the layering of $H'$, group pairs of consecutive layers to produce a layering of $H$. 
Hence, $\ltw(H)\leq 2k$, and $\delta(H)\leq 3 \cdot 2k -1=6k-1$. 
This says that $\rho(G)\leq 6k-1$. 
\end{proof}

Several non-minor-closed graph  classes are known to have bounded layered tree-width \citep{DMW17,DEW17,HKW}. Here is one example. For $g,k\in\NN_0$, a graph is \defn{$(g,k)$-planar} if it has a drawing in a surface of Euler genus at most $g$ such that each edge contains at most $k$ crossings. 
Every $(g,k)$-planar graph $G$ has layered tree-width at most $2(2g+3)(k+1)$ \citep{DEW17}, and thus $\rho(G)\leq 12(2g+3)(k+1)-1$ by \cref{ltw-rho}. 
Also, $G$ is $K_{3,(24k+1)(2g+2)+13}$-subgraph-free~\citep{HW22} and  
$\tw(G)\in O(\sqrt{(g+1)(k+1)|V(G)|})$~\citep{DEW17}. 
Thus the $S=\emptyset$ case of \cref{noKstStructure} with $s=3$ implies:

\begin{cor}
\label{gkPlanarStructure}
There exists $c_1,\dots,c_5\in\NN$ such that for any $g,k\in\NN_0$ every $(g,k)$-planar graph $G$ has a clean $2$-quasi-tree-partition with:
\begin{itemize}
    \item width at most $c_1 (g+1)^3(k+1)^3 \tw(G) \leq c_2 (g+1)^{7/2}(k+1)^{7/2} |V(G)|^{1/2}$, 
    \item degree at most $c_3 (g+1)^3(k+1)^3$, and \item weight at most $c_4 \tw(G) \leq c_5 (g+1)^{1/2}(k+1)^{1/2}|V(G)|^{1/2}$.
\end{itemize}
\end{cor}

Note that \cref{SurfaceCorllary} is the special case $k=0$ of \cref{gkPlanarStructure} (with lightly worse dependence on $g$)

\section{Defective and Clustered Colouring}
\label{ColouringI}

This section presents applications of our structural results from the previous section for graph colouring. 

A \defn{colouring} of a graph $G$ is simply a function $f:V(G)\to\mathcal{C}$ for some set $\mathcal{C}$ whose elements are called \defn{colours}. If $|\mathcal{C}| \leq k$ then $f$ is a \defn{$k$-colouring}. An edge $vw$ of $G$ is \defn{$f$-monochromatic} if $f(v)=f(w)$. A colouring $f$ is \defn{proper} if no edge is $f$-monochromatic. An \defn{$f$-monochromatic component}, sometimes called a \defn{monochromatic component}, is a connected component of the subgraph of $G$ induced by $\{v\in V(G):f(v)=\alpha\}$ for some  colour $\alpha\in \mathcal{C}$. We say $f$ has \defn{clustering} $c$ if every $f$-monochromatic component has at most $c$ vertices. The \defn{$f$-monochromatic degree} of a vertex $v$ is the degree of $v$ in the monochromatic component containing $v$. Then $f$ has \defn{defect} $d$ if every $f$-monochromatic component has maximum degree at most $d$ (that is, each vertex has monochromatic degree at most $d$). 

The \defn{clustered chromatic number $\cchi(\GG)$} of a graph class $\GG$ is the infimum of the set of nonnegative integers $k$ such that for some $c\in\NN$ every graph in $\GG$ has a $k$-colouring with clustering $c$. The \defn{defective chromatic number $\dchi(\GG)$} of a graph class $\GG$ is the infimum of the set of nonnegative integers $k$ such that for some $d\in\NN_0$ every graph in $\GG$ has a $k$-colouring with defect $d$. Every colouring of a graph with clustering $c$ has defect $c-1$. Thus $\dchi(\GG)\leq \cchi(\GG) \leq\bigchi(\GG)$ for every class $\GG$, where \defn{$\bigchi(\GG)$} is the infimum of the set of nonnegative integers $k$ such that every graph in $\GG$ has a proper $k$-colouring.

Clustered and defective colouring have recently been widely studied~\citep{DEMW23,EKKOS15,Wood10,KM07,NSSW19,vdHW18,CL25,KO19,CE19,CLO18,EJ14,EO16,DN17,LO18,HW19,MRW17,LW1,LW2,LW3,LW4,Liu24,NSW22,DS20,EW23,Liu23}; see \citep{WoodSurvey} for a survey. 

\subsection{List Colouring}
\label{ListColouring}

A \defn{list-assignment} for a graph $G$ is a function $L$ that assigns a set $L(v)$ of colours to each vertex $v\in V(G)$. A graph $G$ is \defn{$L$-colourable} if there is a proper colouring of $G$ such that  each vertex $v\in V(G)$ is assigned a colour in $L(v)$.  A list-assignment $L$ is a \defn{$k$-list assignment} if $|L(v)|\geq k$ for each vertex $v\in V(G)$. The \defn{list-chromatic-number $\lchi(G)$} of a graph $G$ is the minimum $k\in\NN_0$ such that $G$ is $L$-colourable for every $k$-list-assignment $L$ of $G$.

For a list-assignment $L$ of a graph $G$ and $d\in\NN_0$, define $G$ to be \defn{$L$-colourable with defect $d$} if there is a  colouring of $G$ with defect $d$ such that  each vertex $v\in V(G)$ is assigned a colour in $L(v)$.  Define $G$ to be \defn{$k$-list-colourable with defect $d$} if $G$ is $L$-colourable with defect $d$ for every $k$-list assignment $L$ of $G$. Similarly, for $c\in\NN$, $G$ is \defn{$L$-colourable with clustering $c$} if there is a  colouring of $G$ with clustering $c$ such that each vertex $v\in V(G)$ is assigned a colour in $L(v)$.  Define $G$ to be \defn{$k$-list-colourable with clustering $c$} if $G$ is $L$-colourable with clustering $c$ for every $k$-list assignment $L$ of $G$. 

The \defn{defective list-chromatic-number} of a graph class $\GG$, denoted by $\ldchi(\GG)$, is the infimum $k\in\NN$ such that for some $d\in\NN$ every graph in $\GG$ is $k$-list-colourable with defect $d$. 
The \defn{clustered list-chromatic-number $\lcchi(\GG)$} of a graph class $\GG$ is the infimum $k\in\NN$ such that for some $c\in\NN$ every graph in $\GG$ is $k$-list-colourable with clustering $c$. 

\citet{OOW19} proved the following result about defective colouring and sparsity. For a graph $G$, let \defn{$\nabla(G)$} be the maximum of $\frac{|E(H)|}{|V(H)|}$, taken over all graphs $H$ such that some $(\leq 1)$-subdivision of $H$ is a subgraph of $G$. (Note that 
$\nabla(G)\leq \rho(G)\leq 2\nabla(G)$ since $\frac{|E(H)|}{|V(H)|} \leq \delta(H) \leq
\frac{2|E(H)|}{|V(H)|}$.)

\begin{thm}[\citep{OOW19}]
\label{OOW}
For any $s,t\in\NN$ and $\nabla\in\mathbb{R}_{>0}$ there exists $d\in\NN$ such that every $K^*_{s,t}$-subgraph-free graph $G$ with $\nabla(G)\leq\nabla$ is $s$-list-colourable with defect $d$.
\end{thm}


A \defn{balanced separator} in a graph $G$ is a set $X\subseteq V(G)$ such that every component of $G-X$ has at most $\frac{|V(G)|}{2}$ vertices. 
For a function $f:\NN\to\mathbb{R}_{\geq 0}$, a graph $G$ \defn{admits $f$-separators} if every induced subgraph $H$ of $G$ has a balanced separator in $H$ of size at most $f(|V(H)|)$. 
A graph class $\GG$ \defn{admits strongly sublinear separators} if there exists a function $f$ with $f(n)\in O(n^{\beta})$ for some fixed $\beta\in[0,1)$ such that every graph in $\GG$ admits $f$-separators. 
If $\GG$ admits strongly sublinear separators, then $\GG$ has bounded $\nabla$  \citep{DN16,ER18}.
\cref{OOW} thus implies:

\begin{cor}
\label{OOWcor}
For every monotone graph class $\GG$ admitting strongly sublinear separators and with  $K^*_{s,t}\not\in\GG$, 
$$\dchi(\GG) \leq \ldchi(\GG) \leq s.$$
\end{cor}

The authors proved the following result in a previous paper. 

\begin{thm}[\citep{LW2}]
\label{TreewidthKst}
For $k,s,t\in\NN$ with $t\geq s$, there exists $c\in\NN$ such that every $K_{s,t}$-subgraph-free graph with tree-width at most $k$ is $(s+1)$-list-colourable with clustering $c$.
\end{thm}

The proof of \cref{TreewidthKst} also works for $K^*_{s,t}$-subgraph-free graphs. Our structure theorem (\cref{noKstStructure}) can be used to prove \cref{TreewidthKst}. In fact, we use \cref{noKstStructure} to prove a more general ``fractional'' version below (\cref{TreewidthKstFractional}). Also note that $s+1$ colours in \cref{TreewidthKst} is best possible~\citep{LW2}. Since \cref{noKstStructure} implies \cref{TreewidthKst}, this says that the ``$(s-1)$-quasi'' term in \cref{Kst-rho,KstBasic,noKstStructure} is also best possible.



Note that a graph is $k$-colourable with clustering $c$ if and only if $G$ has an $H$-partition with width at most $c$ for some graph $H$ with $\chi(H)\leq k$. Thus \cref{TreewidthKst} implies:

\begin{cor}
\label{TreewidthKstPartition}
For $k,s,t\in\NN$ with $t\geq s$, there exists $c\in\NN$ such that every $K_{s,t}$-subgraph-free graph with treewidth at most $k$ there is an $H$-partition of $G$ with width at most $c$, for some graph $H$ with $\chi(H)\leq s+1$.
\end{cor}

Recent works have studied $H$-partitions of bounded width where $\tw(H)$ is small~\citep{ISW24,UTW,DDEHJMMSW24,DvoWoo,DDJMMW}, which is a much stronger property than having $\chi(H)$ small (since $\chi(H)\leq\tw(H)+1$ and there are bipartite graphs with unbounded treewidth (such as grid graphs). So it is natural to ask whether \cref{TreewidthKstPartition} can be strengthened to show that $\tw(H)\leq s$ or just $\tw(H)\leq f(s)$ for some function $f$. \cref{TPW} says this is true for $s=1$. However, \cref{ForceBadH} says it is false for $s\geq 2$: large complete graphs or large complete bipartite graphs are unavoidable in $H$. So in \cref{TreewidthKstPartition} (with $s\geq 2$),  ``$\chi(H)\leq s+1$'' cannot be replaced by ``$\tw(H)\leq f(s)$'' or  ``$\delta(H)\leq f(s)$'' or ``$\lchi(H)\leq f(s)$'' for any function $f$ (by the lower bound of \citet{Alon00}).

\subsection{Fractional Colouring}

Let $G$ be a graph.  For $p,q\in\NN$ with $p\geq q$, a \defn{$p$:$q$-colouring} of $G$ is a function $\phi:V(G)\to \binom{C}{q}$ for some set $C$ with $|C|=p$. That is, each vertex is assigned a set of $q$ colours out of a palette of $p$ colours. For $t\in\RR$, a \defn{fractional $t$-colouring} is a $p$:$q$-colouring for some $p,q\in\NN$ with $\frac{p}{q}\leq t$. A \defn{$p$:$q$-colouring} $\phi$ of $G$ is \defn{proper} if $\phi(v)\cap \phi(w)=\emptyset$ for each edge $vw\in E(G)$. 

The \defn{fractional chromatic number} of $G$ is 
$$\chi_f(G) := \inf\left\{ t \in\RR \,: \, \text{$G$ has a proper fractional $t$-colouring} \right\}.$$
The fractional chromatic number is widely studied; see the textbook \citep{SU97}, which includes a proof of the fundamental property that $\fchi(G)\in\QQ$. 

Fractional 1-defective colourings were first studied by \citet{FS15}; see \citep{GX16,MOS11,Klostermeyer02} for related results. Fractional defective and clustered colouring (with general bounds on the defect and clustering) were introduced by \citet{DS20} and subsequently studied by \citet{NSW22} and \citet{EW23}. For a $p$:$q$-colouring $f:V(G)\to \binom{C}{q}$ of $G$ and for each colour $\alpha\in C$, the subgraph $G[ \{ v \in V(G): \alpha \in f(v) \} ]$ is called an \defn{$f$-monochromatic subgraph} or \defn{monochromatic subgraph} when $f$ is clear from the context. A connected component of an $f$-monochromatic subgraph is called an \defn{$f$-monochromatic component} or \defn{monochromatic component}. Note that $f$ is proper if and only if each $f$-monochromatic component has exactly one vertex. 

A $p$:$q$-colouring has \defn{defect} $d$ if every monochromatic subgraph has maximum degree at most $d$. A $p$:$q$-colouring has \defn{clustering} $c$ if every monochromatic component has at most $c$ vertices. 

The \defn{fractional defective chromatic number $\dfchi(\GG)$} of a graph class $\GG$ is the infimum of all $t>0$ such that, for some $d\in\NN$ every graph in $\GG$ is fractionally $t$-colourable with defect $d$. 
The \defn{fractional clustered chromatic number $\cfchi(\GG)$} of a graph class $\GG$ is the infimum of all $t>0$ such that, for some $c\in\NN$, every graph in $\GG$ is fractionally $t$-colourable with clustering $c$. 

For $k,n\in\NN$, let \defn{$T_{k,n}$} be the rooted tree in which every leaf is at distance $k-1$ from the root, and every non-leaf has $n$ children.  Let \defn{$C_{k,n}$} be the closure of $T_{k,n}$. Colouring each vertex by its distance from the root gives a $k$-colouring of $C_{k,n}$, and any root-leaf path in $C_{k,n}$ induces a $k$-clique. So $\chi(C_{k,n})=k$.
The class $\CC_k:= \{ C_{k,n} : n \in \NN\}$ is important for defective and clustered colouring, and is often called the `standard' example. It is well-known and easily proved (see \citep{WoodSurvey}) that 
\begin{equation}
\label{StandardExample}
\dchi(\CC_k)=\cchi(\CC_k)=\bigchi(\CC_k)=k.
\end{equation}
\citet{NSW22} extended this result (using a result of \citet{DS20}) to the setting of defective and clustered fractional chromatic number by showing that
\begin{equation}
\label{StandardExampleFractional}
\dfchi(\CC_k)= \cfchi(\CC_k) = \bigchi^f(\CC_k) = \dchi(\CC_k)=\cchi(\CC_k)=\bigchi(\CC_k)=k.
\end{equation}
Let $\omega_\Delta(\GG) := \sup\{k\in\NN: |\GG \cap \C_k|=\infty\}$.
Hence \cref{StandardExampleFractional} implies that for every graph class $\GG$, 
$$\cchi(\GG) \geq \max\{\dchi(\GG),\cfchi(\GG)\} \geq \min\{\dchi(\GG),\cfchi(\GG)\} \geq \dfchi(\GG) \geq \omega_\Delta(\GG).$$ 

For every proper minor-closed class $\GG$, \citet{NSW22} showed that
$$\dfchi(\GG)=\cfchi(\GG)=\omega_\Delta(\GG),$$
and \citet{Liu24} strengthened it by showing that $$\dchi(\GG)=\omega_\Delta(\GG).$$


As another example, the result of \citet{NSW22} implies that the class of graphs embeddable in any fixed surface has fractional clustered chromatic number and fractional defective chromatic number 3.

Assuming bounded maximum degree, \citet{Dvorak16} and \citet{DS20} proved the following stronger results\footnote{\cref{DS} is not explicitly stated in \citep{DS20,Dvorak16}, but it can be concluded from 
Lemma~19 in \citep{Dvorak16} (restated as 
Theorem~15 in \citep{DS20}) and Lemma~2 in \citep{DS20} where $f(G)$ is the maximum order of a component of $G$.}:


\begin{thm}[\citep{Dvorak16,DS20}]	 
\label{DS}
Let $f:\NN\to\mathbb{R}_{\geq 0}$ be a function such that $f(n) \in O(n^{\beta})$ for some fixed $\beta\in[0,1)$. 
Then for any $\Delta\in\NN$ and $\epsilon\in\mathbb{R}_{>0}$, there exist $p,q,c\in\NN$ with $p\leq(1+\epsilon)q$ such that every graph of maximum degree at most $\Delta$ admitting $f$-separators is $p$:$q$-colourable with clustering $c$.
\end{thm}

\begin{cor}[\citep{Dvorak16,DS20}]	 
\label{DS-cfchi}
Every hereditary graph class admitting strongly sublinear separators and with bounded maximum degree has fractional clustered chromatic number 1.
\end{cor}

These results lead to the following.

\begin{cor}
\label{SSS-dfchi-cfchi}
For every hereditary graph class $\GG$ admitting strongly sublinear separators, $$\dfchi(\GG)=\cfchi(\GG).$$
\end{cor}

\begin{proof}
It follows from the definitions that $\dfchi(\GG)\leq \cfchi(\GG)$. We now prove that $\cfchi(\GG)\leq \dfchi(\GG)$. 
Fix a graph $G$.
Let $k:= \dfchi(\GG)$.
Thus, for each $\epsilon>0$, there exist $p,q,d \in \NN$ with $p\leq(k+\epsilon)q$ such that $d$ only depends on $\GG$ and $\epsilon$, and $G$ is $p$:$q$-colourable with defect $d$. 
By \cref{DS}, for each $\epsilon'>0$, there exist $p',q',c \in \NN$ with $p'\leq(1+\epsilon')q'$ such that $c$ only depends on $\GG$, $d$ and $\epsilon'$, and every monochromatic subgraph of $G$ (under the first colouring) is $p'$:$q'$-colourable with clustering $c$. 
Taking a product colouring, we find that $G$ is $pp'$:$qq'$-colourable with clustering $c$. 
Now, $pp' \leq (k+\epsilon)(1+\epsilon')qq'$. 
We may choose $\epsilon$ and $\epsilon'$ so that $(k+\epsilon)(1+\epsilon')$ is arbitrarily close to $k$. 
So $\cfchi(\GG)\leq k$.
\end{proof}

\cref{OOWcor,SSS-dfchi-cfchi} immediately imply:

\begin{cor}
\label{SSS-Kst}
If $\GG$ is a monotone graph class admitting strongly sublinear separators and with  $K^*_{s,t}\not\in\GG$, then
$$\cfchi(\GG) = \dfchi(\GG) \leq \dchi(\GG) \leq s.$$
\end{cor}

Note that \cref{SSS-Kst} implies the upper bound of the above-mentioned result of \citet{NSW22} which says that if $\GG$ is the class of graphs with Euler genus $g$, then $\cfchi(\GG) = \dfchi(\GG) =3$ (since $K_{3,2g+3}\not\in\GG$). 
In fact, the following more general result follows from \cref{SSS-Kst} since the class of $(g,k)$-planar graphs is monotone, admits $O(\sqrt{(g+1)(k+1)n})$ separators (see \citep{DEW17}), and does not contain $K_{3,(24k+1)(2g+2)+13}$ (see \citep{HW22}):

\begin{cor} 
\label{gkPlanar}
For any $g,k\in\NN_0$, if $\GG_{g,k}$ is the class of $(g,k)$-planar graphs, then 
$$\cfchi(\GG_{g,k}) = \dfchi(\GG_{g,k}) = \dchi(\GG_{g,k}) =3.$$
\end{cor}

\subsection{Fractional List Colouring}

We now show how to use our structural result about quasi-tree-partitions to strengthen \cref{SSS-Kst} to choosability for graphs with bounded tree-width.

For a list-assignment $L$ of a graph $G$, an \defn{$L$:$q$-colouring} of $G$ is a function $\phi$ such that $\phi(v)$ is a $q$-element subset of $L(v)$ for each vertex $v$ of $G$. 
A graph $G$ is \defn{$p$:$q$-list-colourable with clustering} $c$ if for every $p$-list-assignment $L$ of $G$, there is an $L$:$q$-colouring of $G$ with clustering $c$. 

The following is the main result of this subsection, where the non-fractional ($\ell=1$) case (\cref{TreewidthKst}) was proved by the authors~\citep{LW1}. 

\begin{restatable}{thm}{TreewidthKstFractional}
\label{TreewidthKstFractional}
For any $k,s,t,\ell\in\NN$ there exists $c\in\NN$ such that every $K^*_{s,t}$-subgraph-free graph with tree-width at most $k$ is $(\ell s+1)$:$\ell$-list-colourable with clustering $c$.
\end{restatable}

\cref{TreewidthKstFractional} follows from \cref{noKstStructure} and the next lemma.

\begin{lem}
\label{QuasiTreePartitionFractionalColouring}
For any $s,\ell,k,d\in\NN$, if a graph $G$ has a clean $(s-1)$-quasi-tree-partition of width at most $k$ and degree at most $d$, then $G$ is $(\ell s+1)$:$\ell$-list-colourable with clustering $\max\{\ell k^2, 2k d^{\ell k-1}\}$. 
\end{lem}

\begin{proof}
Let $L$ be an $(\ell s+1)$-list assignment for $G$. We may assume that $L(v)\subseteq \NN$ for each $v\in V(G)$. 
By assumption, there exists $(E_v\subseteq E(G):v\in V(G))$ with $|E_v|\leq s-1$ for each $v\in V(G)$, such that $G-\bigcup_{v\in V(G)}E_v$ has a $T$-partition $(B_x)_{x\in V(T)}$ of width at most $k$ and degree at most $d$, and for each $v\in V(G)$ and for each edge $vw\in E_v$, if $v\in B_x$ and $w\in B_y$, then $y$ is a non-parent ancestor of $x$ in $T$. 
We add edges to $G$ so that $B_x\cup B_y$ is a clique for each edge $xy\in E(T)$. 
To prove this lemma, it suffices to show that $G$ has an $L$:$p$-colouring.

Let $z$ be the root of $T$. 
Let $\preceq_T$ be a total order on $V(T)$ where for all $x,y\in V(T)$, if $\dist_T(z,x)\leq \dist_T(z,y)$ then $x\preceq_T y$. 
Let $\preceq_G$ be a total order on $V(G)$ where for all $v,w\in V(G)$, if $v\in B_x$ and $w\in B_y$ and $x\preceq_T y$, then $v\preceq_G w$. Finally, let $\preceq$ be a total order on $X:=\{(v,i):v\in V(G),i\in L(v)\}$, where for all $(v,i),(w,j)\in X$, if $v\preceq_G w$ then $(v,i)\preceq (w,j)$, and if $v=w$ and $i<j$, then $(v,i)\prec (w,j)$.

We will colour the vertices of $G$ in order of $\preceq_G$. 
For any monochromatic subgraphs $A$ and $B$ of some partially coloured induced subgraph of $G$, we say that $A$ is \defn{older} than $B$ if $(v_A,i_A)\prec(v_B,i_B)$, where for every $C \in \{A,B\}$, $i_C$ is the colour of $C$, and $v_C$ is the minimum vertex in $C$ with respect to $\preceq_G$.

Colour the vertices of $G$ in order of $\preceq_G$, where each vertex $v$ is assigned a set of $\ell$ colours in $L(v)$ distinct from the $\ell$ colours assigned to the end (distinct from $v$) of each edge in $E_v$, and distinct from the colour of the currently oldest monochromatic component adjacent in $G-\bigcup_{u \in V(G)}E_u$ to $v$. 
Such a colouring exists, since $|L(v)|-\ell|E_v|-1\geq |L(v)|-\ell(s-1)-1\geq \ell$. 

Since for each $v\in V(G)$, the ends of each edge in $E_v$ are assigned disjoint sets of colours, if $vw$ is any monochromatic edge of $G$ with $v\in B_x$ and $w\in B_y$ then $x=y$ or $xy\in E(T)$. That is, monochromatic edges of $G$ map to vertices or edges of $T$. This implies that when colouring a vertex $v$, since $B_x\cup B_y$ is a clique for each edge $xy\in E(T)$, no two distinct  pre-existing monochromatic components are merged into one monochromatic component. So colouring a vertex does not change the older relationship between pre-existing monochromatic components. That is, if $C_1$ and $C_2$ are distinct monochromatic components at some point, and $C_1$ is older than $C_2$ at this point, then at any later time, the monochromatic component containing $C_1$ is older than the monochromatic component containing $C_2$. 

Consider an edge $xy\in E(T)$ with $x$ the parent of $y$. Say $C$ is the $j$-th oldest monochromatic component intersecting $B_x$ immediately after all the vertices in $B_x$ have been coloured, where $j\leq \ell |B_x|\leq \ell k$. 
Note that when colouring vertices in $B_y$, $C$ is always contained in the $j$-th oldest monochromatic component intersecting $B_x$ (although this $j$-th oldest monochromatic component intersecting $B_x$ can change from time to time). 
When colouring vertices in $B_y$, by the above colouring procedure, the oldest monochromatic component intersecting $B_x$ cannot intersect $B_y$. 
So after colouring all the vertices in $B_y$, if the monochromatic component $C'$ containing $C$ intersects $B_y$, then $C'$ is the $i$-th oldest monochromatic component intersecting $B_y$, for some $i \leq j-1$. 

Now we bound the clustering of the colouring.
Let $M$ be a monochromatic component when all vertices of $G$ are coloured.
Assume that $M$ intersects all of $B_{x_1},\dots,B_{x_p}$, where $x_i$ is the parent of $x_{i+1}$ for each $i\in\{1,\dots,p\}$.
For each $i \in \{1,\dots,p\}$, let $d_i$ be the smallest integer such that when we just finished colouring all vertices of $x_i$, the $d_i$-th oldest component intersecting $B_{x_i}$ is contained in $M$.
Then $d_1\leq \ell k$, and $d_i\leq d_{i-1}-1$ for each $i\in\{2,\dots,p\}$. 
Thus $1\leq d_p\leq d_1-p+1\leq \ell k-p+1$, implying $p\leq \ell k$. Since $\Delta(T)\leq d$ and $|B_x|\leq  k$ for each $x\in V(T)$, we have $|V(M)|\leq k(1+d+d^2+\dots+d^{\ell k-1})$.
Note that if $d \geq 2$, then $(1+d+d^2+\dots+d^{\ell k-1}) \leq 2d^{\ell k-1}$.
So $|V(M)| \leq \max\{\ell k^2, 2k d^{\ell k-1}\}$. 
\end{proof}

\section{Structural Results II}
\label{StructureII}

This section establishes extensions of the structural results in \cref{StructureI} for  graphs containing no 1-extension of $K_{s,a}$ and no skewered $K_{s,b}$.

\begin{lem} \label{extension-or-skewer}
Let $s,a,b\in\NN$ with $a \geq 2$. For any graph $G$ and set $X \subseteq V(G)$ with $|X|=s$, if some component $C$ of $G-X$ contains at least $(a-1)(b-1)+1$ vertices in $N_G^{\geq s}(X)$, then $G$ contains a 1-extension of $K_{s,a}$ or a skewered $K_{s,b}$.
\end{lem}

\begin{proof}
Let $T$ be a spanning tree of $C$ rooted at a vertex $r$ in $N_G^{\geq s}(X)$. If there exists a path $P$ in $T$ from $r$ containing $b$ vertices in $N_G^{\geq s}(X)$, then $G[X \cup V(P)]$ contains a skewered $K_{s,b}$.
So we may assume that every path in $T$ from $r$ contains at most $b-1$ vertices in $N_G^{\geq s}(X)$. 
In particular, $b \geq 2$.

Define a poset $Q=(N_G^{\geq s}(X) \cap V(C), \preceq)$ such that for any two elements $x,y$ in the ground set, $x \preceq y$ if and only if $x=y$ or $x$ is an ancestor of $y$ in $T$. Since every path in $T$ from $r$ contains at most $b-1$ vertices in $N_G^{\geq s}(X)$, every chain of $Q$ has size at most $b-1$.
By Dilworth's Theorem, $Q$ has an antichain $A$ of size $\lceil |N_G^{\geq s}(X) \cap V(C)|/(b-1) \rceil \geq a$. Since $a \geq 2$, $r \not \in A$. Let $T'$ be the subtree of $T$ consisting of all paths from $r$ to the parents of vertices in $A$. 
Contracting $T'$ into a vertex, and deleting vertices not in $X\cup V(T')\cup A$ gives a 1-extension of $K_{s,a}$.
\end{proof}





If $(B_x)_{x\in V(T)}$ is a partition of a graph $G$ indexed by a rooted tree $T$, then a node $y\in V(T)$ is \defn{$s$-heavy} (with respect to $(B_x)_{x\in V(T)}$) if $|N_G(B_y) \cap \bigcup_{q \in T \uparrow y}B_q| \geq s$. 

\begin{thm}
\label{NoExtensionSkewered-AlmostTreePartition}
Fix $k,s,a,b,\rho \in\mathbb{N}$ with $a,b \geq 2$. 
Let $t:=(s+(a-1)(b-1))(a-1)(b-1)+k+1$. 
Define $c:=c(s,t,\rho)$ as in \cref{pi}.
If $G$ is a graph with $\tw(G)\leq k-1$ and $\rho(G)\leq \rho$ that contains no 1-extension of $K_{s,a}$ and no skewered $K_{s,b}$, then for any set $S\subseteq V(G)$ with $|S| \leq s+(a-1)(b-1)$, there exists a clean $(s-1)$-quasi-tree-partition $(T,(B_x)_{x\in V(T)},(E_v)_{v\in V(G)})$ of $G$ of width at most $18ck$ such that: 
    \begin{itemize}
        \item if $z$ is the root of $T$, then $S\subseteq B_z$, 
        \item every node in $T$ has at most $6c$ $(s+1)$-heavy children, and 
		\item for every $Y \subseteq V(G)$, if $Y$ has at least $\max\{k+1,s+1\}$ common neighbours in $G$, then there exists a vertical path in $T$ passing through each node $x \in V(T)$ with $Y \cap B_x \neq \emptyset$. 
	\end{itemize}
\end{thm}

\begin{proof}
We proceed by induction on $|V(G)|$. 
If $G$ is $K_{s,t}$-subgraph-free, then the result follows from \cref{noKstStructure} (since $|S|\leq s+(a-1)(b-1)\leq t\leq c \leq 12ck$).
This proves the base case. 
Now assume that $G$ contains $K_{s,t}$. 
Hence there exists $X \subseteq V(G)$ with $|X|=s$ with $|N_G^{\geq s}(X)| \geq t$.

By \cref{extension-or-skewer}, every component of $G-X$ contains at most $(a-1)(b-1)$ vertices in $N_G^{\geq s}(X)$.
So $G-X$ contains at least $\lceil \frac{t}{(a-1)(b-1)} \rceil \geq s+(a-1)(b-1)+1 \geq |S|+1$ components.
Hence there exists a component $C$ of $G-X$ disjoint from $S$.
By induction, $G-V(C)$ has a clean $(s-1)$-quasi-tree-partition $(T^1,(B^1_x)_{x\in V(T^1)},(E^1_v)_{v\in V(G-V(C))})$
of width at most $18ck$ such that 
\begin{itemize}
	\item if $z_1$ is the root of $T^1$, then $S\subseteq B^1_{z_1}$, 
		\item every node of $T^1$ has at most $6c$ $(s+1)$-heavy children, and
		\item for every $Y \subseteq V(G)-V(C)$, if $Y$ has at least $\max\{k+1,s+1\}$ common neighbours in $G-V(C)$, then there exists a vertical path in $T^1$ passing through each node $x \in V(T^1)$ with $Y \cap B^1_x \neq \emptyset$.
	\end{itemize}

Let $S' := X \cup (N_G^{\geq s}(X) \cap V(C))$.
Note that $|S'| = |X| + |N_G^{\geq s}(X) \cap V(C)| \leq s+(a-1)(b-1)$ by \cref{extension-or-skewer}. 
By induction applied to $G[V(C) \cup X]$ with $S'$ the specified set, $G[V(C) \cup X]$ has a clean $(s-1)$-quasi-tree-partition 
$(T^2,(B^2_x)_{x\in V(T^2)},(E^2_v)_{v\in V(C) \cup X})$ of width at most $18ck$ such that 
	\begin{itemize}
        \item if $z_2$ is the root of $T^2$ then
 $S'\subseteq B^2_{z_2}$,
		\item every node of $T^2$ has at most $6c$ $(s+1)$-heavy children, and
		\item for every $Y \subseteq X \cup V(C)$, if $Y$ has at least $\max\{k+1,s+1\}$ common neighbours in $G[X \cup V(C)]$, then there exists a vertical path in $T^2$ passing through each node $x \in V(T^2)$ with $Y \cap B^2_x \neq \emptyset$.
	\end{itemize}
Note that $E^2_v=\emptyset$ for each $v \in B^2_{z_2}$, since $z_2$ has no parent in $T^2$. 
Similarly, for each child $x$ of $z_2$, we may assume that $E^2_v=\emptyset$ for each $v \in B^2_x$.

Let $T^3$ be the tree obtained from $T^2$ by adding a new node $z_3$ adjacent to $z_2$.
Consider $T^3$ to be rooted at $z_3$.
Let $B^3_{z_3} := X$, and let $B^3_{z_2} := B^2_{z_2}\setminus X$.
For each $x \in V(T^2)\setminus \{z_2\}$, let $B^3_x := B^2_x$.
For each $x \in V(T^2)$ that is not a child of $z_2$ and for each $v \in B^3_x$, let $E^3_v := E^2_v$.
For each $x \in V(T^2)$ that is a child of $z_2$ and for every $v \in B^3_x$, let $E^3_v$ be the set of edges between $v$ and $X$.
By construction, $(B^3_x)_{x \in V(T^3)}$ is a tree-partition of $G[X \cup V(C)]-\bigcup_{v \in X \cup V(C)}E^3_v$.
Since $N_G^{\geq s}(X) \cap V(C) \subseteq S' \subseteq B^2_{z_2}$, 
$E^3_v$ is a set of at most $s-1$ edges of $G[V(C) \cup X]$ incident with $v$ for every $v \in V(C)$.

Note that $X$ has at least $t$ common neighbours in $G$, and $V(C)$ contains at most $(a-1)(b-1)$ common neighbours of $X$.
So $X$ has at least $t-(a-1)(b-1) = (s-1+(a-1)(b-1))(a-1)(b-1)+k+1 \geq s+k \geq \max\{k+1,s+1\}$ common neighbours in $G-V(C)$.  
Thus, there exists a vertical path $P$ in $T^1$ passing through each node $x \in V(T^1)$ with $X \cap B_x \neq \emptyset$. 
Let $z^*$ be the vertex of $P$ furtherest from $z_1$ such that $X \cap B^1_{z^*} \neq \emptyset$. 
Thus each vertex $x \in V(T^1)$ with $X \cap B^1_x \neq \emptyset$ is an ancestor of $z^*$ or equal to $z^*$.

Let $T$ be the tree obtained from the disjoint union of $T^1$ and $T^3$ by identifying $z^*$ and $z_3$. Consider $T$ to be rooted at $z_1$. So the parent of $z_3$ in $T$ is the parent of $z^*$ in $T^1$.
For each $x \in V(T^1)$, let $B_x := B^1_x$.
For each $x \in V(T^3)-\{z_3\}$, let $B_x := B^3_x$.
Then $(B_x)_{x \in V(T)}$ is a partition of $V(G)$ of width at most $18ck$, and $S \subseteq B_{z_1}$.

For each $v \in V(G)-V(C)$, let $E_v:=E^1_v$.
For each $v \in V(C)-B^2_{z_2}$, let $E_v := E^3_v$.
By construction, for every $v \in V(G)-(B^2_{z_2}-X)$, $E_v$ is a set of at most $s-1$ edges incident with $v$, and for each edge $vw\in E_v$, if $v\in B_x$ and $w\in B_y$, then by the choice of $z^*$, $y$ is a non-parent ancestor of $x$.

For each $v \in B^2_{z_2}-X$, let $E_v$ be the set of edges between $v$ and $X-B^1_{z^*}$.
Since $X \cap B^1_{z^*} \neq \emptyset$, we know $|X-B^1_{z^*}| \leq |X|-1 \leq s-1$, so $v \in N_G^{\leq s-1}(X-B^1_{z^*})$ for every $v \in V(C)$.
Hence $|E_v| \leq s-1$ for every $v \in V(G)$.
And by the choice of $z^*$, for every $v \in B^2_{z_2}-X$ and for each edge $vw\in E_v$, if $v\in B_x$ and $w\in B_y$, then $y$ is a non-parent ancestor of $x$.

By the definition of $E_v$ for $v \in B^2_{z_2}-X$, we know $(B_x)_{x \in V(T)}$ is a $T$-partition of $G-\bigcup_{v \in V(G)}E_v$. Hence $((B_x)_{x\in V(T)},(E_v)_{v\in V(G)})$ is a clean $(s-1)$-quasi-$T$-partition of $G$ of width at most $18ck$.


We now prove the claim about the number of heavy children of a node in $T$. For each $x \in V(T^1)$ and child $y$ of $x$ in $T^1$, if $X \cap B^1_y = \emptyset$, then $N_G(B_y)=N_G(B^1_y) \subseteq V(G)-V(C)$, so $|N_G(B_y) \cap \bigcup_{q \in T \uparrow y}B_q| = |N_G(B^1_y) \cap \bigcup_{q \in T^1 \uparrow y}B^1_q|$; if $X \cap B^1_y \neq \emptyset$, then by the choice of $z^*$, $T \uparrow y$ is disjoint from $V(T^2)$ and hence equal to $T^1 \uparrow y$, so $|N_G(B_y) \cap \bigcup_{q \in T \uparrow y}B_q| = |N_G(B^1_y) \cap \bigcup_{q \in T^1 \uparrow y}B^1_q|$.
If $x \in V(T^1)$ and $y$ is a child of $x$ in $T$ but not in $T^1$, then $x=z^*$ and $y=z_2$, so $N_G(B_y) \subseteq X \cup V(C)$, and hence $N_G(B_y) \cap \bigcup_{q \in T \uparrow y} \subseteq X$, which has size $s < s+1$.
And for every $x \in V(T^2)$ and child $y$ of $x$ in $T^2$, we have $N_G(B_y) \subseteq X \cup V(C)$ and $X \subseteq B^2_{z_2}$ and $z_2 \in T^2 \uparrow y$, so $N_G(B_y) \cap \bigcup_{q \in T \uparrow y}B_q \subseteq N_{G[X \cup V(C)]}(B^2_y) \cap \bigcup_{q \in T^2 \uparrow y}B^2_q$.
Therefore, every node of $T$ has at most $6c$ $(s+1)$-heavy children.

Let $Y$ be an arbitrary subset of $V(G)$ such that $Y$ has at least $\max\{k+1,s+1\}$ common neighbours in $G$.
If $Y \cap V(C) \neq \emptyset \neq Y-(X \cup V(C))$, then $X$ contains all common neighbours of $Y$, but there are at least $s+1 >|X|$ common neighbours of $Y$, a contradiction.
So either $Y \subseteq V(G)-V(C)$ or $Y \subseteq X \cup V(C)$.
If $Y \subseteq X$, then the vertical path in $T$ from $z_1$ to $z^*$ contains all nodes $x \in V(T)$ with $B_x \cap Y \neq \emptyset$ by the definition of $z^*$.
So we may assume $Y \not \subseteq X$.
Hence, if $Y \subseteq V(G)-V(C)$, then all common neighbours of $Y$ in $G$ are contained in $V(G)-V(C)$; if $Y \subseteq X \cup V(C)$, then all common neighbours of $Y$ in $G$ are contained in $V(C) \cup X$.
Hence there exists a vertical path in $T$ containing all nodes $x \in V(T)$ with $Y \cap B_x \neq \emptyset$.
This completes the proof.
\end{proof}

The value $t$ (and hence $c$) in the previous theorem depends on $k$. 
We can make it independent of $k$ if drop the `clean' requirement. The proof is almost identical to the previous theorem, but we rewrite it for completeness. 

\begin{thm} 
\label{s-2-no-k}
Fix $s,a,b,\rho \in\NN$ with $a \geq 2$ and $b \geq 2$. 
Let $t:=(s+(a-1)(b-1))(a-1)(b-1)+1$. 
Define $c:=c(s,t,\rho)$ as in \cref{pi}.
For any $k\in\NN$, if $G$ is a graph with $\tw(G)\leq k-1$ and $\rho(G)\leq\rho$ that contains no 1-extension of $K_{s,a}$ and no skewered $K_{s,b}$, then for any set $S\subseteq V(G)$ with $|S| \leq s+(a-1)(b-1)$, there exists an $(s-1)$-quasi-tree-partition $(T,(B_x)_{x\in V(T)},(E_v)_{v\in V(G)})$ of $G$ of width at most $18ck$, such that:
    \begin{itemize}
        \item if $z$ is the root of $T$ then $S\subseteq B_z$, and 
        \item every node of $T$ has at most $6c$ $(s+1)$-heavy children.
    \end{itemize}
\end{thm}

\begin{proof}
We proceed by induction on $|V(G)|$.
If $G$ is $K_{s,t}$-subgraph-free, then the result follows from \cref{noKstStructure} (since $|S|\leq s+(a-1)(b-1)\leq t\leq c \leq 12ck $). 
This proves the base case. Now assume that $G$ contains $K_{s,t}$. Hence there exists $X \subseteq V(G)$ with $|X|=s$ with $|N_G^{\geq s}(X)| \geq t$.

By \cref{extension-or-skewer}, every component of $G-X$ contains at most $(a-1)(b-1)$ vertices in $N_G^{\geq s}(X)$.
So $G-X$ contains at least $\lceil \frac{t}{(a-1)(b-1)} \rceil \geq s+(a-1)(b-1)+1 \geq |S|+1$ components.
Hence there exists a component $C$ of $G-X$ disjoint from $S$.
By induction, $G-V(C)$ has an $(s-1)$-quasi-tree-partition 
$(T^1,(B^1_x)_{x\in V(T^1)},(E^1_v)_{v\in V(G)\setminus V(C)})$ of width at most $18ck$, such that if $z_1$ is the root of $T^1$ then $S\subseteq B^1_{z_1}$, and every node of $T^1$ has at most $6c$ $(s+1)$-heavy children.

Let $S' := X \cup (N_G^{\geq s}(X) \cap V(C))$.
Note that $|S'| = |X| + |N_G^{\geq s}(X) \cap V(C)| \leq s+(a-1)(b-1)$
By induction applied to $G[V(C) \cup X]$ with specified set $S'$, there exists an $(s-1)$-quasi-tree-partition $(T^2,(B^2_x)_{x\in V(T_2)},(E^2_v)_{v\in V(C)\cup X})$, such that if $z_2$ is the root of $T^2$ then $S'\subseteq B^2_{z_2}$, and every node of $T$ has at most $6c$ $(s+1)$-heavy children. 
Note that $E^2_v=\emptyset$ for every $v \in B^2_{z_2}$, since $z_2$ has no parent in $T^2$. Similarly, for every child $x$ of $z_2$, we may assume $E^2_v=\emptyset$ for every $v \in B^2_x$.

Let $T^3$ be the tree obtained from $T^2$ by adding a new node $z_3$ adjacent to $z_2$.
Consider $T^3$ to be rooted at $z_3$.
Let $B^3_{z_3} := X$ and $B^3_{z_2} := B^2_{z_2}-X$. 
For each $x \in V(T^2)-\{z_2\}$, let $B^3_x := B^2_x$.
For each $x \in V(T^2)$ that is not a child of $z_2$ and for each $v \in B^3_x$, let $E^3_v := E^2_v$.
For each $x \in V(T^2)$ that is a child of $z_2$ and for every $v \in B^3_x$, let $E^3_v$ be the edges between $v$ and $X$.
By construction, $(B^3_x)_{x \in V(T^3)}$ is a tree-partition of $G[X \cup V(C)]-\bigcup_{v \in X \cup V(C)}E^3_v$.
Since $N_G^{\geq s}(X) \cap V(C) \subseteq S' \subseteq B^2_{z_2}$, $E^3_v$ is a set of at most $s-1$ edges of $G[V(C) \cup X]$ incident with $v$ for every $v \in V(C)$.

Let $z^*$ be a node of $T^1$ with $B^1_{z^*} \cap X \neq \emptyset$, and maximising the distance in $T^1$ between $z^*$ and $z_1$. 
Let $T$ be the tree obtained from the disjoint union of $T^1$ and $T^3$ by identifying $z^*$ and $z_3$.
Note that $T$ is rooted at $z_1$, and the parent of $z_3$ in $T$ is the parent of $z^*$ in $T^1$.
For each $x \in V(T^1)$, let $B_x := B^1_x$.
For each $x \in V(T^3)-\{z_3\}$, let $B_x := B^3_x$.
Then $(B_x)_{x \in V(T)}$ is a partition of $V(G)$ with $|B_x| \leq 18ck$ for every $x \in V(T)$, and $S \subseteq B_{z_1}$.

For each $v \in V(G)-V(C)$, let $E_v := E^1_v$.
For each $v \in V(C)-B^2_{z_2}$, let $E_v := E^3_v$.
By construction, for each $v \in V(G)-(B^2_{z_2} \setminus X)$, $E_v$ is a set of at most $s-1$ edges incident with $v$, and for each edge $vw\in E_v$, if $v\in B_x$ and $w\in B_y$, then by the choice of $z^*$, $y \in T \uparrow x$. 

For each $v \in B^2_{z_2}-X$, let $E_v$ be the set of edges between $v$ and $X-B^1_{z^*}$.
Since $X \cap B^1_{z^*} \neq \emptyset$, we know $|X-B^1_{z^*}| \leq |X|-1 \leq s-1$, so $v \in N_G^{\leq s-1}(X-B^1_{z^*})$ for each $v \in V(C)$.
Hence $|E_v| \leq s-1$ for each $v \in V(G)$.
And by the choice of $z^*$, for every $v \in B^2_{z_2}-X$ and for each edge $vw\in E_v$, if $v\in B_x$ and $w\in B_y$, then $y \in T \uparrow x$. 

By the definition of $E_v$ for $v \in B^2_{z_2}-X$, we know $(B_x)_{x \in V(T)}$ is a tree-partition of $G-\bigcup_{v \in V(G)}E_v$.

For each $x \in V(T^1)$ and child $y$ of $x$ in $T^1$, if $X \cap B^1_y = \emptyset$, then $N_G(B_y)=N_G(B^1_y) \subseteq V(G)-V(C)$, so $|N_G(B_y) \cap \bigcup_{q \in T \uparrow y}B_q| = |N_G(B^1_y) \cap \bigcup_{q \in T^1 \uparrow y}B^1_q|$; if $X \cap B^1_y \neq \emptyset$, then by the choice of $z^*$, $T \uparrow y$ is disjoint from $V(T^2)$ and hence equal to $T^1 \uparrow y$, so $|N_G(B_y) \cap \bigcup_{q \in T \uparrow y}B_q| = |N_G(B^1_y) \cap \bigcup_{q \in T^1 \uparrow y}B^1_q|$.
If $x \in V(T^1)$ and $y$ is a child of $x$ in $T$ but not in $T^1$, then $x=z^*$ and $y=z_2$, so $N_G(B_y) \subseteq X \cup V(C)$, and hence $N_G(B_y) \cap \bigcup_{q \in T \uparrow y} \subseteq X$, which has size $s < s+1$.
And for every $x \in V(T^2)$ and child $y$ of $x$ in $T^2$, we have $N_G(B_y) \subseteq X \cup V(C)$ and $X \subseteq B^2_{z_2}$ and $z_2 \in T^2 \uparrow y$, so $N_G(B_y) \cap \bigcup_{q \in T \uparrow y}B_q \subseteq N_{G[X \cup V(C)]}(B^2_y) \cap \bigcup_{q \in T^2 \uparrow y}B^2_q$.
Therefore, for every $x \in V(T)$, there are at most $6c$ children $y$ of $x$ in $T$ such that $|N_G(B_y) \cap \bigcup_{q \in T \uparrow y}B_q| \geq s+1$.
%
This completes the proof.
\end{proof}

\cref{s-2-no-k} is applicable with $\rho(G)\leq\tw(G)$ by \cref{rho-tw}, in which case $c(s,t,\rho) \leq \max\{t,t \cdot \tw(G)^{s-1}+1\} \leq t (\tw(G)+1)^{s-1}$. The next result follows, which implies \cref{NoExtensionSkewered}. 

\begin{cor} 
\label{PseudoTW}
Fix $s,a,b \in\NN$ with $a,b \geq 2$. 
Let $t:=(s+(a-1)(b-1))(a-1)(b-1)+1$. For any $k\in\NN$, if $G$ is a graph with $\tw(G)\leq k-1$ that contains no 1-extension of $K_{s,a}$ and no skewered $K_{s,b}$, then $G$ has an $(s-1)$-quasi-tree-partition $(T,(B_x)_{x\in V(T)},(E_v)_{v\in V(G)})$ of width at most $18 tk^{s}$, such that every node has at most $6tk^{s-1}$ $(s+1)$-heavy children.
\end{cor}

\cref{s-2-no-k} leads to the following strengthening of \cref{NoExtensionSkewered-TPW} and the known result that graphs with bounded tree-width and maximum degree have bounded tree-partition width (\cref{TPW}).

\begin{cor} 
For any $a,b.k\in\NN$ with $a,b \geq 2$, if $G$ is a graph with $\tw(G)\leq k-1$ that contains no 1-extension of $K_{1,a}$ and no skewered $K_{1,b}$, then $G$ has a tree-partition $(T,(B_x)_{x\in V(T)})$ of width at most $18a^2b^2k$, such that for every $x \in V(T)$, there are at most $6a^2b^2$ children $y$ of $x$ such that $|N_G(B_y) \cap B_x| \geq 2$.
\end{cor}

\begin{proof}
By assumption, \cref{s-2-no-k} is applicable with $s=1$ and $\rho=k$. 
Let $(T,(B_x)_{x\in V(T)},(E_v)_{v\in V(G)})$ be the $0$-quasi-tree-partition of $G$ obtained. 
Since $s=1$, the width of $(T,(B_x)_{x\in V(T)},(E_v)_{v\in V(G)})$ is at most $18ck$, where $c \leq (1+(a-1)(b-1))(a-1)(b-1)+1 \leq ((a-1)(b-1)+1)^2 \leq a^2b^2$.
Since $(T,(B_x)_{x\in V(T)},(E_v)_{v\in V(G)})$ is a $0$-quasi-tree-partition, $E_v=\emptyset$ for every $v \in V(G)$. 
Thus $(B_x)_{x \in V(T)}$ is a tree-partition of $G$, and for any $x \in V(T)$ and child $y$ of $x$, $N_G(B_y) \cap \bigcup_{q \in T \uparrow x}B_q = N_G(B_y) \cap B_x$. Since each node of $T$ has at most $6c \leq 6a^2b^2$ 2-heavy children, the final claim follows.
\end{proof}

\section{Colouring II}
\label{ColouringII}

This section proves colouring results for graphs with no 1-extension of $K_{s,a}$ and no skewered $K_{s,b}$. The next lemma enables this. 

\begin{lem} 
\label{non-fractional-tree-colouring}
Fix $d,r \in \NN_0$ and $w \in\NN$. If $G$ is a graph that has an $r$-quasi-tree-partition $(T,(B_x)_{x\in V(T)},(E_v)_{v\in V(G)})$ of width at most $w$ such that every node of $T$ has at most $d$ $(r+2)$-heavy children, then $G$ is $(r+2)$-list-colourable with clustering $w(d+1)^{w}$. 
\end{lem}

\begin{proof}
We may add edges to $G$ so that $B_x$ is a clique for each $x \in V(T)$, and for each edge $xy$ of $E(T)$, 
if $y$ is an $(r+2)$-heavy child of $x$, then every vertex in $B_x$ is adjacent in $G$ to every vertex in $B_y$.

Let $\leq_T$ be a BFS-ordering of $V(T)$.
Let $\leq_G$ be a linear ordering of $V(G)$ such that for any distinct $x,y \in V(T)$ with $x \leq_T y$, if $u \in B_x$ and $v \in B_y$, then $u \leq_G v$.
For any $S \subseteq V(G)$, let \defn{$m(S)$} be the smallest vertex in $S$ according to $\leq_G$.
For each subgraph $H$ of $G$, let $\defn{m(H)}:=m(V(H))$.
For any $S_1,S_2$, where each $S_i$ is a subset of $V(G)$ or a subgraph of $G$, we define $S_1 \leq_G S_2$ if and only if $m(S_1) \leq_G m(S_2)$.
Denote $V(T)$ by $\{t_1,t_2,\dots,t_{|V(T)|}\}$ such that $t_i \leq_T t_j$ for every $i \leq j$.

Let $L$ be an arbitrary $(r+2)$-list-assignment of $G$.

Define $f_1$ to be an $L|_{B_{t_1}}$-colouring of $G[B_{t_1}]$ by defining $f_1(v)$ to be an arbitrary element in $L(v)$ for every $v \in B_{t_1}$.
For every $i \in [|V(T)|]-\{1\}$, define an $L|_{\bigcup_{j=1}^iB_{t_j}}$-colouring $f_i$ of $G[\bigcup_{j=1}^iB_{t_j}]$ as follows:
	\begin{itemize}
		\item For each $v \in \bigcup_{j=1}^{i-1}B_{t_j}$, let $f_i(v) := f_{i-1}(v)$.
		\item Let $p_i$ be the parent of $t_i$ in $T$.
		\item Let $M_i$ be the $f_{i-1}$-monochromatic component intersecting $B_{p_i}$ with the smallest $m(M_i)$. 
		\item Let $a_i$ be the colour of $M_i$.
		\item If $|N_G(B_{t_i}) \cap \bigcup_{q \in T \uparrow t_i}B_q| \geq r+2$, then for every $v \in B_{t_i}$, let $f_i(v)$ be an element in $L(v)-\{a_i, f_{i-1}(u): uv \in E_v\}$.
		\item If $|N_G(B_{t_i}) \cap \bigcup_{q \in T \uparrow t_i}B_q| \leq r+1$, then for every $v \in B_{t_i}$, let $f_i(v)$ be an element in $L(v)-\{f_{i-1}(u): u \in N_G(B_{t_i}) \cap \bigcup_{q \in T \uparrow t_i}B_q\}$.
	\end{itemize}
Since $|L(v)| \geq r+2$, $f_i(v)$ can be defined as above.

Let $f=f_{|V(T)|}$.
It suffices to show that $f$ has clustering $c := w(d+1)^{w}$.

Let $M$ be an arbitrary $f$-monochromatic component.
Let $S_M = \{t \in V(T): B_t \cap V(M) \neq \emptyset\}$.
To prove $|V(M)| \leq c$, it suffices to show that $S_M$ induces a subtree $T_M$ with maximum degree at most $d+1$ (\cref{non-fractional-tree-colouring-claim3} below) such that every path in $T_M$ from the root of $T_M$ to any other vertex contains at most $w$ vertices (implied by \cref{non-fractional-tree-colouring-claim5} below). This will imply that $|V(M)| \leq w|S_M| \leq w\sum_{i=0}^{w-1}(d+1)^i \leq w(d+1)^{w}= c$, as desired. 

\begin{claim} 
\label{non-fractional-tree-colouring-claim1}
If $uv$ is an edge of $G$ with $\{u,v\} \subseteq V(M)$ and $u \leq_G v$, then there exists an edge $xy$ of $T$ such that $\{u,v\} \subseteq B_x \cup B_y$.
\end{claim} 

\begin{proof} 
Let $x,y \in V(T)$ such that $u \in B_x$ and $v \in B_y$.
Since $u \leq_G v$, $x \leq_T y$.
Suppose to the contrary that $x \neq y$ and $x$ is not the parent of $y$.
Since $(B_t: t \in V(T))$ is a tree-partition of $G-\bigcup_{q \in V(G)}E_q$, $uv \in E_u \cup E_v$.
By the property of $E_u$ and $E_v$, we know that $uv \in E_v$ and $x \in T \uparrow y$.
Note that $x \in T \uparrow y$ implies that $u \in N_G(B_y) \cap \bigcup_{q \in T \uparrow y}B_q$.
If $|N_G(B_y) \cap \bigcup_{q \in T \uparrow y}B_q| \geq r+2$, then since $uv \in E_v$, $f(u) \neq f(v)$ by the definition of $f$.
If $|N_G(B_y) \cap \bigcup_{q \in T \uparrow y}B_q| \leq r+1$, then since $u \in N_G(B_y) \cap \bigcup_{q \in T \uparrow y}B_q$, $f(u) \neq f(v)$ by the definition of $f$.
Hence $f(u) \neq f(v)$ in either case, contradicting that $M$ is monochromatic.
\end{proof}

Since $M$ is connected, by \cref{non-fractional-tree-colouring-claim1}, we know $T_M = T[S_M]$ is a subtree of $T$.

\begin{claim} 
\label{non-fractional-tree-colouring-claim2}
For every edge $ty$ of $T_M$, where $t$ is the parent of $y$, $|N_G(B_y) \cap \bigcup_{q \in T \uparrow y}B_q| \geq r+2$.
\end{claim}

\begin{proof}
Since $\{t,y\} \subseteq S_M$, there exist $u \in B_t \cap V(M)$ and $v \in B_y \cap V(M)$.
Since $M$ is connected, there exists a path $P$ in $M$ between $u$ and $v$.
We choose $u$ and $v$ such that $P$ can be chosen to be as short as possible.
By \cref{non-fractional-tree-colouring-claim1}, $uv \in E(G)$.
Since $t \in T \uparrow y$, $u \in N_G(B_y) \cap \bigcup_{q \in T \uparrow y}B_q$. 
So if $|N_G(B_y) \cap \bigcup_{q \in T \uparrow y}B_q| \leq r+1$, then $f(u) \neq f(v)$, a contradiction.
\end{proof}

\begin{claim} 
\label{non-fractional-tree-colouring-claim3}
$T_M$ has maximum degree at most $d+1$.
\end{claim}

\begin{proof}
Suppose to the contrary that there exists $t \in V(T_M)$ with degree at least $d+2$.
So $t$ has at least $d+1$ children in $T_M$.
By the property of $(B_x)_{x \in V(T)}$, there exists a child $y$ of $t$ in $T_M$ such that $|N_G(B_y) \cap \bigcup_{q \in T \uparrow y}B_q| \leq r+1$, contradicting \cref{non-fractional-tree-colouring-claim2}.
\end{proof}


\begin{claim} 
\label{non-fractional-tree-colouring-claim4}
For every vertical path $P$ in $T_M$, there exists a path $P_M$ in $M$ such that $|V(P_M) \cap B_t|=1$ for every $t \in V(P)$, and $V(P_M) \cap B_t=\emptyset$ for every $t \in V(T)-V(P)$.
\end{claim} 

\begin{proof}
By \cref{non-fractional-tree-colouring-claim2}, for every edge $ty$ of $P$, where $t$ is a parent of $y$, $|N_G(B_y) \cap \bigcup_{q \in T \uparrow y}B_q| \geq r+2$, so $B_t \cup B_y$ is a clique of $G$ by assumption.
Then the claim follows.
\end{proof}

\begin{claim} 
\label{non-fractional-tree-colouring-claim5}
For any $k\in\NN$, if there exists a vertical path $P$ in $T_M$ on at least $k$ vertices, then there exist at least $k-1$ $f$-monochromatic components $Q$ with $m(Q)<_Gm(M)$ and $V(Q) \cap B_{t_P} \neq \emptyset$, where $t_P$ is the root of $P$.
\end{claim}

\begin{proof}
We prove this claim by induction on $k$.
The case $k=1$ is obvious.
So we may assume $k \geq 2$.

By \cref{non-fractional-tree-colouring-claim4}, there exists a path $P_M$ in $M$ such that $|V(P_M) \cap B_t|=1$ for every $t \in V(P)$, and $V(P_M) \cap B_t=\emptyset$ for every $t \in V(T)-V(P)$.
Let $y$ be the end of $P$ other than $t_P$ ($y$ exists since $k \geq 2$).
Let $v_P$ be the vertex in $V(P_M) \cap B_{y}$.
Let $p$ be the parent of $y$.
Let $i \in [|V(T)|]$ such that $y=t_i$.
Note that there exists a $f_{i-1}$-monochromatic component $M'$ containing $P_M-v_P$.
Since $M' \subseteq M$, $m(M) \leq_G m(M')$.
By considering the path in $T_M$ from the root of $T_M$ to $p$, we know $m(M)=m(M')$ by \cref{non-fractional-tree-colouring-claim2,non-fractional-tree-colouring-claim4}.

By \cref{non-fractional-tree-colouring-claim2}, 
$|N_G(B_y) \cap \bigcup_{q \in T \uparrow y}B_q| \geq r+2$.
Since $V(M) \cap B_y \neq \emptyset$, $a_i$ is not the colour of $M$.
That is, there exists a $f_{i-1}$-monochromatic component $C_i$ intersecting $B_p$ with $m(C_i)<_G m(M')=m(M)$.
Let $C$ be the $f$-monochromatic component containing $C_i$.
So $m(C) \leq_G m(C_i)<_G m(M)$.
Applying 
\cref{non-fractional-tree-colouring-claim1,non-fractional-tree-colouring-claim2,non-fractional-tree-colouring-claim4} to $C$, we know $m(C)=m(C_i)$, and the root $r_C$ of the subtree $T_C$ induced by $\{t \in V(T): B_t \cap V(C) \neq \emptyset\}$ is an ancestor of the root $r_M$ of $T_M$ or equal to $r_M$.
Hence there exists a vertical path in $T_C$ from $t_P$ to $p$ on at least $k-1$ vertices, and $V(C) \cap B_{t_P} \neq \emptyset$.
By the induction hypothesis, there exist at least $k-2$ $f$-monochromatic components $Q$ with $m(Q)<_G m(C)$ and $V(Q) \cap B_{t_P} \neq \emptyset$.
By collecting those $Q$ together with $C$, we obtain at least $k-1$ $f$-monochromatic components $Q$ with $m(Q)<_G m(M)$ and $V(Q) \cap B_{t_P} \neq \emptyset$.
\end{proof}

Let $r_M$ be the root of $T_M$.
Since $|B_x| \leq w$ for every $x \in V(T)$, there are at most $w-1$ $f$-monochromatic components $Q$ with $m(Q)<_Gm(M)$ and $V(Q) \cap B_{r_M} \neq \emptyset$.
By 
\cref{non-fractional-tree-colouring-claim5}, every vertical path in $T_M$ contains at most $w$ vertices.
This completes the proof. 
\end{proof}

If the tree of the tree-partition has bounded maximum degree, then we can do fractional colouring that reduces the ratio for the number of colours. The proof is very similar, but we write it again for completeness.

\begin{lem} 
\label{fractional-tree-colouring}
For any integers $w,\ell \geq 1$ and $d,r \geq 0$, if a graph $G$ has an $r$-quasi-tree-partition $(T,(B_x)_{x\in V(T)},(E_v)_{v\in V(G)})$ of width at most $w$ and degree at most $d$, then $G$ is $((r+1)\ell+1)$:$\ell$-list-colourable with clustering $wd^{w}$. 
\end{lem}

\begin{proof}
We may add edges to $G$ so that $B_x$ is a clique for every $x \in V(T)$, and for any edge $xy$ of $E(T)$, every vertex in $B_x$ is adjacent in $G$ to every vertex in $B_y$.

Let $\leq_T$ be a BFS-ordering of $V(T)$.
Let $\leq_G$ be a linear ordering of $V(G)$ such that for any distinct $x,y \in V(T)$ with $x \leq_T y$, if $u \in B_x$ and $v \in B_y$, then $u \leq_G v$.
For any $S \subseteq V(G)$, let $m(S)$ be the smallest vertex in $S$ according to $\leq_G$.
For any subgraph $H$ of $G$, let $m(H)=m(V(H))$.
For any $S_1,S_2$, where each $S_i$ is a subset of $V(G)$ or a subgraph of $G$, we define $S_1 \leq_G S_2$ if and only if $m(S_1) \leq_G m(S_2)$.
Denote $V(T)$ by $\{t_1,t_2,...,t_{|V(T)|}\}$ such that $t_i \leq_T t_j$ for every $i \leq j$.

Let $L$ be an arbitrary $((r+1)\ell+1)$-list-assignment of $G$.

Define $f_1$ to be an $L|_{B_{t_1}}$-colouring of $G[B_{t_1}]$ by defining $f_1(v)$ to be an arbitrary subset of $L(v)$ with size $\ell$ for every $v \in B_{t_1}$.
For every $i \in [|V(T)|]-\{1\}$, we define an $L|_{\bigcup_{j=1}^iB_{t_j}}$-colouring $f_i$ of $G[\bigcup_{j=1}^iB_{t_j}]$ as follows:
	\begin{itemize}
		\item For each $v \in \bigcup_{j=1}^{i-1}B_{t_j}$, let $f_i(v) := f_{i-1}(v)$.
		\item Let $p_i$ be the parent of $t_i$ in $T$.
		\item Let $M_i$ be the $f_{i-1}$-monochromatic component intersecting $B_{p_i}$ with the smallest $m(M_i)$. 
		\item Let $a_i$ be the colour of $M_i$.
		\item For each $v \in B_{t_i}$, let $f_i(v)$ be a subset of $L(v)-(\{a_i\} \cup \bigcup_{uv \in E_v}f_{i-1}(u))$ with size $\ell$.
	\end{itemize}
Since $|E_v| \leq r$ and $|L(v)| \geq (r+1)\ell+1$, $f_i(v)$ can be defined as above.

Let $f=f_{|V(T)|}$.
It suffices to show that $f$ has clustering $c := wd^{w}$.

Let $M$ be an arbitrary $f$-monochromatic component.
Let $S_M = \{t \in V(T): B_t \cap V(M) \neq \emptyset\}$.
To prove $|V(M)| \leq c$, it suffices to show that $S_M$ induces a subtree $T_M$ such that every path in $T_M$ from the root of $T_M$ to any other vertex contains at most $w$ vertices (since it implies that $|V(M)| \leq w|S_M| \leq w\sum_{i=0}^{w-1}d^i \leq wd^{w}=c$, as $\Delta(T) \leq d$).

\begin{claim}
\label{fractional-tree-colouring-claim1}
If $uv$ is an edge of $G$ with $\{u,v\} \subseteq V(M)$ and $u \leq_G v$, then there exists an edge $xy$ of $T$ such that $\{u,v\} \subseteq B_x \cup B_y$.
\end{claim}

\begin{proof}
Let $x,y \in V(T)$ such that $u \in B_x$ and $v \in B_y$.
Since $u \leq_G v$, $x \leq_T y$.
Suppose to the contrary that $x \neq y$ and $x$ is not the parent of $y$.
Since $(B_t: t \in V(T))$ is a tree-partition of $G-\bigcup_{q \in V(G)}E_q$, $uv \in E_u \cup E_v$.
By the property of $E_u$ and $E_v$, we know that $uv \in E_v$ and $x \in T \uparrow y$.
Since $uv \in E_v$, $f(u) \neq f(v)$ by the definition of $f$, contradicting that $M$ is monochromatic.
\end{proof}

Since $M$ is connected, by \cref{fractional-tree-colouring-claim1}, we know $T_M = T[S_M]$ is a subtree of $T$.

\begin{claim} 
\label{fractional-tree-colouring-claim2}
For every vertical path $P$ in $T_M$, there exists a path $P_M$ in $M$ such that $|V(P_M) \cap B_t|=1$ for every $t \in V(P)$, and $V(P_M) \cap B_t=\emptyset$ for every $t \in V(T)-V(P)$.
\end{claim} 

\begin{proof}
Since for every edge $ty$ of $P$, $B_t \cup B_y$ is a clique of $G$ by assumption, the claim follows.
\end{proof}

\begin{claim} 
\label{fractional-tree-colouring-claim3}
For any $k\in\NN$, if there exists a vertical path $P$ in $T_M$ on at least $k$ vertices, then there exist at least $k-1$ $f$-monochromatic components $Q$ with $m(Q)<_Gm(M)$ and $V(Q) \cap B_{t_P} \neq \emptyset$, where $t_P$ is the root of $P$.
\end{claim} 

\begin{proof}
We prove this claim by induction on $k$.
The case $k=1$ is obvious.
So we may assume $k \geq 2$.

By \cref{fractional-tree-colouring-claim2}, 
there exists a path $P_M$ in $M$ such that $|V(P_M) \cap B_t|=1$ for every $t \in V(P)$, and $V(P_M) \cap B_t=\emptyset$ for every $t \in V(T)-V(P)$.
Let $y$ be the end of $P$ other than $t_P$ ($y$ exists since $k \geq 2$).
Let $v_P$ be the vertex in $V(P_M) \cap B_{y}$.
Let $p$ be the parent of $y$.
Let $i \in [|V(T)|]$ such that $y=t_i$.
Note that there exists a $f_{i-1}$-monochromatic component $M'$ containing $P_M-v_P$.
Since $M' \subseteq M$, $m(M) \leq_G m(M')$.
By considering the path in $T_M$ from the root of $T_M$ to $p$, we know $m(M)=m(M')$. 

Since $V(M) \cap B_y \neq \emptyset$, $a_i$ is not the colour of $M$.
That is, there exists a $f_{i-1}$-monochromatic component $C_i$ intersecting $B_p$ with $m(C_i)<_G m(M')=m(M)$.
Let $C$ be the $f$-monochromatic component containing $C_i$.
So $m(C) \leq_G m(C_i)<_G m(M)$.
Applying \cref{fractional-tree-colouring-claim1} and \cref{fractional-tree-colouring-claim2} to $C$, we know $m(C)=m(C_i)$, and the root $r_C$ of the subtree $T_C$ induced by $\{t \in V(T): B_t \cap V(C) \neq \emptyset\}$ is an ancestor of the root $r_M$ of $T_M$ or equal to $r_M$.
Hence there exists a vertical path in $T_C$ from $t_P$ to $p$ on at least $k-1$ vertices, and $V(C) \cap B_{t_P} \neq \emptyset$.
By the induction hypothesis, there exist at least $k-2$ $f$-monochromatic components $Q$ with $m(Q)<_G m(C)$ and $V(Q) \cap B_{t_P} \neq \emptyset$.
By collecting those $Q$ together with $C$, we obtain at least $k-1$ $f$-monochromatic components $Q$ with $m(Q)<_G m(M)$ and $V(Q) \cap B_{t_P} \neq \emptyset$.
\end{proof}

Let $r_M$ be the root of $T_M$.
Since $|B_x| \leq w$ for every $x \in V(T)$, there are at most $w-1$ $f$-monochromatic components $Q$ with $m(Q)<_Gm(M)$ and $V(Q) \cap B_{r_M} \neq \emptyset$.
By \cref{fractional-tree-colouring-claim3}, every vertical path in $T_M$ contains at most $w$ vertices.
This proves the lemma.
\end{proof}

\begin{thm}
For any $s,a,b,k \in \NN$ with $a,b \geq 2$, there exists $c\in\NN$ such that for any $\ell\in\NN$ and graph $G$ with $\tw(G) \leq k-1$:
    \begin{enumerate}
        \item If $G$ contains no 1-extension of $K_{s,a}$ and no skewered $K_{s,b}$, then $G$ is $(s+1)$-list-colourable with clustering $c$.
        \item If $G$ is $K_{s,a}$-subgraph-free, then $G$ is $(s\ell+1)$:$\ell$-list-colourable with clustering $c$.
    \end{enumerate}
\end{thm}

\begin{proof}
Recall that $\rho(G) \leq \tw(G)$ by \cref{rho-tw}. 
Statement~1 follows from \cref{s-2-no-k} and \cref{non-fractional-tree-colouring}.
Statement~2 follows from \cref{heart} and \cref{fractional-tree-colouring}. 
\end{proof}



\noindent{\bf Acknowledgement:} 
This paper was partially written when the first author visited the National Center for Theoretical Sciences in Taiwan. He thanks the National Center for Theoretical Sciences for its hospitality.

{
\fontsize{10pt}{11pt}
\selectfont
\bibliographystyle{DavidNatbibStyle}
\bibliography{DavidBibliography}
}
\end{document}